\documentclass[12pt,leqno,twoside]{amsart}
\usepackage{amssymb,amsmath,amsthm,soul,color}

\usepackage{t1enc}
\usepackage[cp1250]{inputenc}
\usepackage{a4,indentfirst,latexsym}
\usepackage{graphics}
\usepackage{mathrsfs}
\usepackage{cite,enumitem,graphicx}
\usepackage[colorlinks=true,urlcolor=blue,
citecolor=red,linkcolor=blue,linktocpage,pdfpagelabels,
bookmarksnumbered,bookmarksopen]{hyperref}
\usepackage[english]{babel}
\usepackage[left=2.61cm,right=2.61cm,top=2.72cm,bottom=2.72cm]{geometry}
\usepackage[metapost]{mfpic}
\usepackage[normalem]{ulem}
\usepackage{cancel}

\newtheorem{Th}{Theorem}[section]
\newtheorem{Prop}[Th]{Proposition}
\newtheorem{Lem}[Th]{Lemma}
\newtheorem{Cor}[Th]{Corollary}

\newtheorem{Rem}[Th]{Remark}

\newenvironment{altproof}[1]
{\noindent
	{\em Proof of {#1}}.}
{\nopagebreak\mbox{}\hfill $\Box$\par\addvspace{0.5cm}}

\newcommand{\wt}{\widetilde}

\newcommand{\vp}{\varphi}

\newcommand{\eps}{\varepsilon}

\def\div{\mathop{\mathrm{div}\,}}

\def\Z{\mathbb{Z}}

\def\N{\mathbb{N}}
\def\R{\mathbb{R}}

\def\dim{\mathrm{dim}}

\def\cl{\mathrm{cl\,}}

\def\tU{{\tilde U}}

\def\W{\mathcal{W}}

\newcommand{\cC}{{\mathcal C}}

\newcommand{\cH}{{\mathcal H}}

\newcommand{\cJ}{{\mathcal J}}
\newcommand{\cK}{{\mathcal K}}
\newcommand{\cL}{{\mathcal L}}
\newcommand{\cM}{{\mathcal M}}
\newcommand{\cN}{{\mathcal N}}

\newcommand{\cT}{{\mathcal T}}

\newcommand{\cV}{{\mathcal V}}
\newcommand{\cW}{{\mathcal W}}

\renewcommand{\dim}{{\rm dim}\,}

\newcommand{\ga}{\gamma}

\newcommand{\Ga}{\Gamma}

\newcommand{\tw}{\widetilde{w}}

\def\curlop{\nabla\times}
\newcommand{\weakto}{\rightharpoonup}
\newcommand{\pa}{\partial}

\newcommand{\tu}{\widetilde{u}}
\newcommand{\tv}{\widetilde{v}}

\newcommand{\tPhi}{\widetilde{\Phi}}

\newcommand{\talpha}{\widetilde{\alpha}}
\newcommand{\tbeta}{\widetilde{\beta}}

\newcommand{\cTto}{\stackrel{\cT}{\longrightarrow}}
\DeclareMathOperator{\Id}{Id}
\DeclareMathOperator{\dom}{dom}

\DeclareMathOperator*{\essinf}{ess\,inf}
\DeclareMathOperator*{\esssup}{ess\,sup}

\newcommand{\dist}{\text{\rm dist}}

\newcommand{\cnabla}{\overset{\circ}{\nabla}}

\numberwithin{equation}{section}

\allowdisplaybreaks

\begin{document}
	\title[Travelling waves for Maxwell's equations]{Travelling waves for Maxwell's equations in nonlinear and nonsymmetric media}
\author{Jaros\l{}aw Mederski}
\address{J. Mederski \hfill\break  
	Institute of Mathematics,\hfill\break 
	Polish Academy of Sciences  \hfill\break
	ul. \'Sniadeckich 8, 00-656 Warsaw, Poland
	}
\email{jmederski@impan.pl}

\author{Wolfgang Reichel}
\address{W. Reichel \hfill\break 
	Departement of Mathematics, Institute for Analysis,\hfill\break
	Karlsruhe Institute of Technology (KIT), \hfill\break
	D-76128 Karlsruhe, Germany}
\email{wolfgang.reichel@kit.edu}

\date{\today}



\begin{abstract}
We look for travelling wave fields 
$$
E(x,y,z,t)= U(x,y) \cos(kz+\omega t)+ \wt U(x,y)\sin(kz+\omega t),\quad (x,y,z)\in\R^3,\, t\in\R
$$ 
satisfying Maxwell's equations in a nonlinear  medium which is not necessarily cylindrically symmetric. The nonlinearity of the medium enters Maxwell's equations by postulating a nonlinear material law $D=\eps E+\chi(x,y, \langle |E|^2\rangle)E$ between the electric field $E$, its time averaged intensity $\langle |E|^2\rangle$ and the electric displacement field $D$. We derive a new semilinear elliptic problem for the profiles $U,\tU:\R^2\to\R^3$
$$Lu-V(x,y)u=f(x,y,u)\quad\hbox{with }u=\begin{pmatrix} U \\ \wt U \end{pmatrix}, \hbox{ for }(x,y)\in\R^2,$$
where $f(x,y,u)=\omega^2\chi(x,y, |u|^2)u$. Solving this equation we can obtain exact travelling wave solutions of the underlying nonlinear Maxwell equations. We are able to deal with super quadratic and subcritical focusing \-effects, e.g. in the Kerr-like materials with the nonlinear susceptibility of the form $\chi(x,y,\langle |E^2|\rangle E) = \chi^{(3)}(x,y)\langle |E|^2\rangle E$. A variational approach is presented for the semilinear problem. The energy functional associated with the equation is strongly indefinite, since $L$ contains an infinite dimensional kernel. The methods developed in this paper  may be applicable to other strongly indefinite elliptic problems and other nonlinear phenomena.
\end{abstract}

\maketitle

\section*{Introduction}
\setcounter{section}{1}

We are looking for travelling wave fields 
\begin{equation}\label{eq:travel_wave}
E(x,y,z,t)= U(x,y) \cos(kz+\omega t)+ \wt U(x,y)\sin(kz+\omega t)
\end{equation}
solving in the absence of charges and currents the Maxwell system $\nabla\times E+\partial_t B=0$ (Faraday's law), $\nabla\times B=\partial_t D$ (Amp\'{e}re's law) together with $\div D=0$ and $\div B=0$. We require the linear magnetic material law $B=\mu(x,y)H$ and the nonlinear electric material law $D=\eps(x,y) E+\chi(x,y, \langle |E|^2\rangle)E$ where $\langle |E(x,y,z)|^2\rangle = \frac{1}{T}\int_0^T |E(x,y,z,t)|^2\,dt$ is the average intensity of a time-harmonic electric field over one period $T=2\pi/\omega$. Taking the curl of Faraday's law and inserting the material laws for $B$ and $D$ together with Amp\'{e}re's law we find that $E$ has to satisfy the {\em nonlinear electromagnetic wave equation}
\begin{equation} \label{newe}
\nabla\times\Big(\frac{1}{\mu(x,y)}\nabla \times E\Big) + \partial_{tt}\left(\epsilon(x,y) E + \chi(x,y, \langle |E|^2\rangle)E\right)=0\hbox{ for }(x,y,z,t)\in\R^3\times\R.
\end{equation}
Here $U,\wt U:\R^2\to \R^3$ are the profiles of the travelling waves, $\omega>0$ is the temporal frequency and $k\in \R\setminus\{0\}$ the spatial wave number in the direction of propagation,
$\eps=\eps(x,y)$ is the permittivity of the medium, $\mu=\mu(x,y)$ is the magnetic permeability, and $\chi$ is the scalar nonlinear susceptibility which depends  on $(x,y)$ and on the time averaged intensity of $E$ only. 

Note that having solved the nonlinear electromagnetic wave equation, one obtains the electric displacement field $D$ directly from the constitutive relation
\begin{equation} \label{susceptibility}
D=\eps(x,y) E+\chi(x,y, \langle |E|^2\rangle)E.
\end{equation}
Moreover the magnetic induction
$B$ may be obtained by time integrating Faraday's law with divergence free initial conditions,  and finally the magnetic field $H$ is given by 
$H=\mu^{-1} B$.  Altogether, we find {\em exact propagation} of the electromagnetic field in the nonlinear medium according to the Maxwell equations  with the time-averaged material law \eqref{susceptibility},  see also \cite{Stuart91,Stuart:1993,FundPhotonics}.

In physics and mathematical literature there are several simplifications relying on approximations of  the nonlinear electromagnetic wave equation. The most prominent one is the scalar or vector nonlinear Schr\"odinger equation. For instance, one assumes that the term $\nabla(\div E)$ in $\curlop(\curlop E)= \nabla(\div E)-\Delta E$ is negligible and can be dropped, or one can use the so-called slowly varying envelope approximation. However, this approach may produce non-physical solutions; see e.g. \cite{Akhmediev-etal, Ciattoni-etal:2005} and references therein. Therefore,  in this paper, we are interested in exact travelling wave solutions of the Maxwell equations.

 We would like to mention that exact propagation of travelling waves of the nonlinear electromagnetic wave problem \eqref{newe} have been studied analytically so far only in cylindrically symmetric media. Namely, if $E$ is an axisymmetric $TE$-mode of the form 
\begin{equation*}
E(x,y,z,t)= U(x,y) \cos(kz+\omega t)
\end{equation*}
with 
\begin{equation}\label{eq:Symmetric}
U(x,y):=u(r)(-y/r,x/r,0),\hbox{ where }r=\sqrt{x^2+y^2}
\end{equation}
and the scalar function $u$ only depends on $r$, then solutions of \eqref{newe}  have been considered in a series of papers by Stuart and Zhou \cite{StuartZhou96,Stuart91,StuartZhou05,StuartZhou10,StuartZhou03,StuartZhou01,Stuart04} for asymptotically  constant  susceptibilities and by McLeod,  Stuart and Troy \cite{McLeodStuartTroy} for a cubic nonlinear polarization. Clearly, $E$ has the form of \eqref{eq:travel_wave} with $\tU=0$.
The search for these solutions reduces to a one-dimensional varia\-tional problem or an ODE for $u(r)$, which simplifies the problem considerably. However, if the medium is not necessarily cylindrically symmetric, then it is not clear how to find travelling waves \eqref{eq:travel_wave} with $\tU=0$ analytically and whether any variational approach can be provided with this constraint.

The aim of this work is to provide an analysis of the travelling waves of the general form \eqref{eq:travel_wave} propagating in media, which are not necessarily cylindrically symmetric. To the best of our knowledge it is the first analytical study of travelling waves of \eqref{newe}, i.e., the Maxwell system with material law \eqref{susceptibility}, where the medium is not supposed to be cylindrically symmetric. We present a variational approach which allows to treat \eqref{newe} and to find  ground states solutions of the problem with the least possible energy as well as infinitely many geometrically distinct bound states.  

Let us briefly comment on the problem of finding time-harmonic  standing wave solutions of \eqref{newe} of the form 
$$E(x,y,z,t)= U(x,y,z) \cos(\omega t).$$ 
This leads to the so-called  nonlinear {\em curl-curl problem} and has been recently studied e.g. in \cite{BartschMederski,BartschMederskiJFA} on a bounded domain and in \cite{BDPR,MederskiENZ,MederskiSchinoSzulkin} on $\R^3$, see also the survey \cite{BartschMederskiSurvey} and references therein. In the curl-curl problem, however,  $U$  is required to be localized in all space directions, i.e., it is supposed to
lie in some Lebesgue space over $\R^3$. Since travelling waves of the form \eqref{eq:travel_wave} are not localized, they have not been taken into account in these works. We would like to also mention that the study of time-harmonic standing waves in $\R^3$ in the nonsymmetric case  has been presented only in \cite{MederskiENZ,MederskiSchinoSzulkin}. The methods used there required assumptions about the vanishing properties of the permittivity $\eps$, or even $\eps=0$, and the double power behaviour of the nonlinear effect, e.g. $\chi(x,y,z,E)= \Gamma(x)\min\{|\langle E \rangle|^{p-2},|\langle E \rangle|^{q-2}\}$ with $2<p<6<q$. Note that $p=4$ corresponds to the Kerr-type effect, but only for sufficiently strong fields $|\langle E\rangle |>1$. In this work, however, we are able to treat the probably most common type of nonlinearity in the physics and engineering literature, the {\em Kerr nonlinearity}
\begin{equation*}
\chi(x,y,\langle |E^2|\rangle) E = \chi^{(3)}(x,y)\langle |E|^2\rangle E
\end{equation*}
and  we no longer require that the permittivity vanishes at infinity.

The search for travelling waves of the form \eqref{eq:travel_wave} leads to a new nonlinear elliptic problem. Namely, in order to solve the nonlinear  electromagnetic wave equation, we observe that
the profiles $U, \wt U$ of the travelling wave satisfy the elliptic problem 
\begin{equation}\label{eq_m}
L \begin{pmatrix} U \\ \wt U \end{pmatrix} -\omega^2\epsilon(x,y) \begin{pmatrix} U \\ \wt U \end{pmatrix}= \omega^2\chi\left(x,y, \frac{1}{2}(|U|^2+|\wt U|^2)\right) \begin{pmatrix} U \\ \wt U \end{pmatrix},
\end{equation}
where
$$ L = \begin{pmatrix} 
-\partial_{yy}+k^2 & \partial_{xy} & 0  & 0 & 0 & k\partial_x \\
\partial_{xy} & -\partial_{xx}+k^2 & 0 & 0 & 0 & k\partial_y \\
0 & 0 & -\partial_{xx}-\partial_{yy} & k\partial_x & k \partial_y & 0 \\
0 & 0 & -k\partial_x & -\partial_{yy}+k^2 & \partial_{xy} & 0  \\
0 & 0 & -k\partial_y & \partial_{xy} & -\partial_{xx}+k^2 & 0 \\
-k\partial_x & -k\partial_y & 0 & 0 & 0 &-\partial_{xx} - \partial_{yy} 
\end{pmatrix}.
$$
For simplicity we have assumed that the magnetic permeability is a constant given by $\mu=1$.
Let us define 
$$
\dom(L):= \left\{\begin{pmatrix} U \\ \wt U \end{pmatrix} \in L^2(\R^2)^6: L\begin{pmatrix} U \\ \wt U \end{pmatrix} \in L^2(\R^2)^6\right\}
$$
where $L((U,\wt U)^T)$ is defined in the sense of distributions.  One verifies that
 the second-order differential operator $L:\dom(L) \subset L^2(\R^2)^6 \to L^2(\R^3)^6$ is elliptic and self-adjoint, see Section~\ref{sec:varsetting} for details.
 
Our aim is to find solutions $u:\R^2\to\R^6$ to the following slight generalization of \eqref{eq_m} given by
\begin{equation}\label{eq}
Lu-V(x)u=f(x,u)\quad\hbox{for }x\in\R^2
\end{equation}
involving the operator $L$, where we assume $f(x,u)=\partial_u F(x,u)$.  From now on the space variable in $\R^2$ will be denoted by $x$ instead of $(x,y)$.

Observe that, if 
\begin{equation}\label{eq:application}
V(x)=\omega^2\eps(x),\quad F(x,u)=\omega^2\chi\Big(x,\frac{1}{2}|u|^2\Big),\quad x\in \R^2, u\in\R^6
\end{equation}
then \eqref{eq} leads to \eqref{eq_m} and we obtain the {\em exact} propagation of the travelling electromagnetic waves $(E,B)$, where $E$ is given by \eqref{eq:travel_wave} and $B$ is provided by Faraday's law and a subsequent time integration. On the other hand, \eqref{eq} is more general, since the nonlinear potential $F$ may depend on the direction of $u$ and not necessarily on $|u|=\sqrt{U^2+\wt U^2}$ as in anisotropic media.

Before we describe our results in more detail, let us comment on the main difference of the current work and previous works \cite{Stuart04,Stuart91,Stuart:1993,StuartZhou01,StuartZhou96,StuartZhou03} in cylindrically symmetric media. In these previous works the reduction of the curl-curl-operator to the Laplacian when acting on $E(x,y,z,t)=U(x,y)\cos(kz+\omega t)$ was possible by building-in the constraint $\div(U)=0$ into the ansatz as in \eqref{eq:Symmetric}. In our work we admit two profiles $U,\tilde U$ in the ansatz \eqref{eq:travel_wave} and thus reduce the curl-curl-operator to the operator $L$ given above. If we would obtain a one-profile solution $(U,0)$ of \eqref{eq_m} with $k\not =0$ then automatically $\div(U)=0$ follows. However, finding such a one-profile solution might be restricted to cylindrically symmetric media. Two-profile solutions like in our ansatz \eqref{eq:travel_wave} allow us to treat more general non-cylindrically symmetric media and a large class of nonlinearities.

The model nonlinearity which we have in mind is 
\begin{equation}\label{Ex:model}
F(x,u)=\frac{1}{p}\Gamma(x)|u|^p
\end{equation}
with $\Gamma\in L^{\infty}(\R^2)$ positive, bounded away from $0$ and $p>2$. In particular if $p=4$ we deal with the Kerr nonlinearity.

 We want to find weak solutions of \eqref{eq} which correspond to critical points of the following functional
\begin{equation}\label{eq:action}
J(u):=\frac12 b_L(u,u)
-\frac12\int_{\R^2}V(x)|u|^2\, d x- \int_{\R^2}F(x,u) \, dx
\end{equation}
defined on a Banach space $X\subset L^2(\R^2)^6\cap  L^p(\R^2)^6$ given later. Here $b_L(\cdot,\cdot)$ is the bilinear form associated with $L$ such that $b_L(u,\vp)= \int_{\R^2}\langle Lu,\vp\rangle\, dx$ for all $u\in \dom(L)$ and all $\vp\in \cC_0^\infty(\R^2)^6$.

Now let us enlist several difficulties underlying the problem.  For $\alpha,\talpha\in\cC_0^{\infty}(\R^2)$ and for $\beta,\tbeta\in \cC_0^\infty(\R^2)^3$
let us denote
\begin{equation*}
\begin{array}{ll} \cnabla \begin{pmatrix} \alpha\\ \talpha  \end{pmatrix}  := \begin{pmatrix} 
\partial_x \alpha \\
\partial_y \alpha \\
k \talpha \\
\partial_x \talpha \\
\partial_y \talpha \\
-k \alpha 
\end{pmatrix} :\R^2\to\R^6,  \qquad
\cnabla\times \begin{pmatrix} \beta\\ \tbeta  \end{pmatrix} := \begin{pmatrix}
\partial_y \beta_3- k\tbeta_2 \\
k\tbeta_1-\partial_x \beta_3 \\
\partial_x\beta_2-\partial_y \beta_1 \\
\partial_y\tbeta_3+k\beta_2 \\
-k\beta_1-\partial_x\tbeta_3\\
\partial_x\tbeta_2-\partial_y\tbeta_1
\end{pmatrix} :\R^2\to\R^6
\end{array}
\end{equation*}
and observe that $L=\cnabla\times\cnabla\times$, $\cnabla\times\cnabla=0$ and thus 
$w:=\cnabla \begin{pmatrix} \alpha\\ \talpha  \end{pmatrix}\in \ker(L)$. Moreover $J$ may be unbounded from above and from below and its critical points may have infinite Morse index. This is due to the infinite-dimensional kernel of $L$. 
In addition to these problems related to the strongly indefinite geometry of $J'$, we also have to deal with the lack of compactness issues. Namely, the functional $J$ is not (sequentially) weak-to-weak$^*$ continuous, i.e. the weak convergence $u_n\weakto u$ in $X$ does not imply that $J'(u_n)\weakto J'(u)$ in $X^*$.
In particular, if $J'(u_n)\to 0$, we do not know whether $u$ is a critical point of $J$ and solves \eqref{eq}.

In Theorem \ref{th:L} we show that the spectrum $\sigma(L)=\{0\}\cup [k^2,\infty)$, where $0$ is an eigenvalue of infinite multiplicity and $[k^2,\infty)$ consists of absolutely continuous spectrum. This allows us to consider the following general assumptions.

\begin{itemize}
	\item[(V)] $V\in L^{\infty}(\R^2)$ and $0<\essinf V\leq\esssup V<k^2$. Moreover there is $V_0\in \R$ such that $V-V_0\in L^{\frac{p}{p-2}}(\R^2)$ for some $p>2$.
	\item[(F1)] $F:\R^2\times\R^2\to [0,\infty)$ is differentiable with respect to the second variable $u\in\R^2$, and $f=\pa_uF:\R^2\times\R^2\to\R^2$ is measurable in $x\in\R^2$, continuous in $u\in\R^2$ for a.e. $x\in\R^2$. Moreover $f$ is $\Z^2$-periodic in $x$ i.e. $f(x,u)=f(x+y,u)$ for $x,u\in \R^2$, $y\in\Z^2$.
	\item[(F2)] $|f(x,u)|=o(1)$ as $u\to 0$ uniformly in $x\in\R^2$.
	\item[(F3)] There are $p>2$ and $c_1>0$ such that
		$$|f(x,u)|\le c_1(1+|u|^{p-1})\quad\hbox{for all  }u\in\R^2\hbox{ and a.e. }x\in\R^2$$
	\item[(F4)] There is $c_2>0$ such that
	$$\liminf_{|u|\to\infty}F(x,u)/|u|^p\geq c_2\quad\hbox{for a.e. }x\in\R^2.$$
\end{itemize}

Assumptions (V), (F1)--(F3) are sufficient to show that $J$ is of class $\cC^1$. (F4) provides a lower bound estimate for large $|u|$, however we will need also its stronger variant
\begin{itemize}
	\item[(F4')] There is $c_2>0$ such that
$$F(x,u)\geq c_2 |u|^p\quad\hbox{for all  }u\in\R^2\hbox{ and a.e. }x\in\R^2.$$
\end{itemize}

In order to deal with ground states one has to assume the following  assumption
\begin{itemize}
	\item[(F5)] If $ \langle f(x,u),v\rangle = \langle f(x,v),u\rangle >0$, then
	$\ \displaystyle F(x,u) - F(x,v)
	\le \frac{\langle f(x,u),u\rangle^2-\langle f(x,u),v\rangle^2}{2\langle f(x,u),u\rangle}$.\\
	Moreover $\langle f(x,u),u\rangle\geq 2F(x,u)$ for every $x\in\R^2$ and $u\in\R^2$.
\end{itemize}
Condition (F5) was introduced in \cite{BartschMederski,MederskiSystem}. Conditions (F4) and (F5) imply in particular the convexity of $F$ in $u$, cf. Remark~\ref{rem:F5}. Observe that if $F$ is isotropic in $u$, i.e. $F(x,u)=\chi(x,\frac12|u|^2)$, $\chi(x,s)\geq 0$, $\partial_s\chi(x,s)$ is nondecreasing and $\chi(x,0)=0$, then (F5) is satisfied. In general, (F5) does not imply the Ambrosetti-Rabinowitz condition (F6) given below, cf. \cite{AR}. For instance, if we take $0<a<b$ then we may find a function $\chi$ of class $\cC^1$ such that $\partial_s\chi(s)=s^{\frac{p-2}{2}}$ for $0\leq s<a$, $\partial_s\chi(s)$ is constant for $s\in [a,b]$, and $\partial_s\chi(s)=cs^{\frac{p-2}{2}}$ for $s>b$ and a suitable $c>0$. Then (F1)--(F5)  are satisfied, but the Ambrosetti-Rabinowitz condition (F6) does not hold. Note also that if $f$ and $\wt f$ satisfy (F1)--(F4), then also $\alpha f+\beta \wt f$ satisfies (F1)--(F4) for $\alpha,\beta>0$. It is not clear if (F5) considered alone has this positive additivity property, however similarly as in \cite[Remark 3.3 (b)]{MederskiSystem} we check that if $f$ and $\wt f$ satisfy (F1)--(F5) simultaneously, then $\alpha f+\beta \wt f$ satisfies the same assumptions. 

\medskip

The first main result reads as follows.
\begin{Th}\label{th:main} Suppose that  (V) and (F1)--(F5) hold.
	If $V=V_0$, or $V(x)>V_0$ for a.e. $x\in\R^3$ and (F4') holds, then there is a nontrivial solution to \eqref{eq} of the form $u_0=v+w$ with $v\in H^1(\R^2)^2\setminus\{0\}$ and $w\in L^2(\R^2)^2\cap L^p(\R^2)^2$ such that $Lw=0$. Moreover $u_0$ is a ground state solution, i.e. $J(u_0)=\inf_{\cN}J>0$, where
		\begin{equation*}
		\cN:=\Big\{u\in X\setminus\{0\}: J'(u)(u)=0\hbox{ and } J'(u)\Big(\cnabla \begin{pmatrix} \alpha\\ \talpha  \end{pmatrix} \Big)=0 \hbox{ for any }\alpha,\talpha\in\cC_0^{\infty}(\R^2)\Big\}.
	\end{equation*}
If, in addition, $f$ is odd in $u$ and $V=V_0$, then there is an infinite sequence $(u_n)$ of $\Z^2$-distinct solutions, i.e. $(\Z^2\ast u_n)\cap (\Z^2\ast u_m)=\emptyset$ for $n\neq m$, where
$\Z^2\ast u:=\{u(\cdot+z):z\in\Z^2\}$ for $u\in X$.
\end{Th}

Note that $\cN$ contains all nontrivial critical points and $\cN$ is contained in the classical Nehari manifold $\{u\in X\setminus\{0\}: J'(u)(u)=0\}$, see \cite{Nehari2} and Section \ref{sec:criticaslpoitth} for further properties. Condition (F5) is important  to obtain a bounded Palais-Smale sequence for $J$ at level  $\inf_{\cN} J$. However it is not clear how to obtain a bounded Palais-Smale sequence for $J$ under (F1)--(F4) and the classical Ambrosetti-Rabinowitz condition
\begin{itemize}
	\item[(F6)] There is $\gamma>2$ such that for every $u\in\R^2$ and a.e. $x\in\R^2$
	$$\langle f(x,u),u\rangle\geq \gamma F(x,u).$$
\end{itemize}
Instead, we need to consider (F4') together with (F6) in order to prove the boundedness of Palais-Smale sequences,  see Lemmas \ref{ineq:AR} and \ref{lem:bounded}.

\begin{Th}\label{th:main2} Suppose that $F$ is convex in $u\in\R^2$, (F1)--(F3), (F4'), (F6) and that (V) holds for a constant function $V$, i.e., $V(x)=V_0$ with $0<V_0<k^2$. Then there is a nontrivial solution to \eqref{eq} of the form $u_0=v+w$ with $v\in H^1(\R^2)^2\setminus\{0\}$ and $w\in L^2(\R^2)^2\cap L^p(\R^2)^2$ such that $Lw=0$. Moreover $u_0$ is a least energy solution, i.e. $J(u_0)=c_0$, where
$$c_0:=\inf\big\{J(u): J'(u)=0\hbox{ and }u\in X\setminus\{0\}\big\}>0.$$
If, in addition, $f$ is odd in $u$, then there is an infinite sequence $(u_n)$ of $\Z^2$-distinct solutions.
\end{Th}

We show that if $u\in X$ solves \eqref{eq}, then the {\em total electromagnetic energy} per unit interval in $x_3$-direction given by
 \begin{eqnarray}\label{eq:EM_Energy}
 \cL(t):=\frac12\int_{\R^2}\int_{a}^{a+1}\langle E,D\rangle +\langle B,H\rangle\,dx_3\, d(x_1, x_2)
 \end{eqnarray}
 is finite; see Corollary \ref{Prop_energy}. We do not know, however,  whether the fields $E$, $D$, $B$ and $H$ are localized, i.e. decay to zero as $|(x_1,x_2)|\to\infty$, however $X$ lies in $L^2(\R^2)^6\cap L^p(\R^2)^6$. Therefore $E$ has a weaker decay property and cannot travel in the $(x_1,x_2)$-plane. The finiteness of the total electromagnetic energy and the localization problem attract a strong attention in the study of self-guided beams of light in nonlinear media; see e.g. \cite{Stuart91,Stuart:1993}. 
 
 The first crucial step in our approach is to build the functional and variational setting for the new operator $L$ and the problem \eqref{eq}, which will be demonstrated in the next Section~\ref{sec:varsetting}. Next we recognize the strongly indefinite nature of \eqref{eq:action} and show that it is of the form of the critical point theory presented in \cite{BartschMederski,BartschMederskiJFA,MederskiSchinoSzulkin} and built for curl-curl problems. However, we work in a different functional setting and under a different set of assumptions, so that we have to slightly refine the results, in particular we do not always assume condition (I8) given in Section~\ref{sec:criticaslpoitth}. In Section \ref{sec:PS} we deal with the lack of compactness issue, in particular with the lack of  weak-to-weak$^*$ continuity of $J'$. By means of a profile decomposition result (Theorem \ref{ThSplit}), we are able to prove this regularity in some weakly closed topological constraint $\cM\subset X$, so that a weak limit point of a bounded Palais-Smale sequence of $\cM$ is a critical point of $J$. In the last Section \ref{sec:proof} we show that a variant of Cerami's condition holds. In particular we show that any Cerami sequence is bounded, and we emphasize that the proof the boundedness is considerably nonstandard and not straightforward, even if (F6) holds.
 Some technical inequalities have to be worked out, see details in Lemma \ref{ineq:AR} and Lemma \ref{lem:bounded}. Finally we complete the proof of Theorems \ref{th:main} and \ref{th:main2}.

\section{Variational setting}\label{sec:varsetting}
We introduce the following notation. If $u=\begin{pmatrix} U\\ \tU \end{pmatrix}\in\R^6$, then $|u|:=\big(|U|^2+|\tU|^2\big)^{1/2}$ and $\langle\cdot,\cdot \rangle$ denotes the Euclidean inner product in $\R^N$, $N\geq 1$. In the sequel $\langle\cdot,\cdot\rangle_2$ denotes the inner product in $L^2(\R^2)^6$ and $|\cdot|_q$ denotes the usual $L^q$-norm for $q\in [1,+\infty]$. Furthermore we denote by $C$ a generic positive constant which may vary from one inequality to the next. We always assume that $k\neq 0$.

Let us introduce the following space
\begin{eqnarray*}
	\cV&:=&\Big\{u=\begin{pmatrix} U \\ \wt U \end{pmatrix} \in H^1(\R^2)^6: \Big\langle u,\cnabla \begin{pmatrix} \alpha\\ \talpha  \end{pmatrix}\Big\rangle_2=0 \mbox{ for any }\alpha,\talpha\in\cC_0^{\infty}(\R^2)\Big\}\\
	&=&\Big\{u=\begin{pmatrix} U \\ \wt U \end{pmatrix} \in H^1(\R^2)^6: \partial_{x_1} u_1 +\partial_{x_2} u_2 +k \tu_3=0, \partial_{x_1} \tu_1 +\partial_{x_2} \tu_2 -ku_3=0 \mbox{ a.e. in }\R^2\Big\}
\end{eqnarray*}
and note that it is a closed subspace of $H^1(\R^2)^6$. Let us consider the following norm in  $\cV$
$$\|u\|:=\Big(\sum_{i=1}^3|\nabla u_i|^2_2+k^2|u_i|^2_2+|\nabla \tu_i|^2_2+k^2|\tu_i|^2_2\Big)^{1/2}$$
which is equivalent to the standard $H^1$-norm. Let $\cW$ be the completion of 
vector fields $w=\cnabla\begin{pmatrix} \Phi\\ \tPhi \end{pmatrix}$, $\Phi,\tPhi\in\cC_0^{\infty}(\R^2)^3$ with respect to the following norm
$$\|w\|:=\big(|w|_2^2+|w|_p^2\big)^{1/2}$$
so that $\cW\subset L^2(\R^2)^6\cap L^p(\R^2)^6$.
Note that $\cV\cap\cW=\{0\}$ and we may define a norm on $\cV\oplus\cW$ as follows
$$\|v+w\|=\|v\|+\|w\|,\quad v\in\cV, w\in\cW.$$

\begin{Th}\label{eq:Helmholtz}
The spaces $\cV$ and $\cW$ are closed subspaces of $L^2(\R^2)^6$ and orthogonal with respect to $\langle \cdot,\cdot\rangle_2$ and $X:=\cV\oplus\cW$ is the completion of  $\cC_0^{\infty}(\R^2)^6$ with respect to the norm $\|\cdot\|$.
\end{Th}
\begin{proof}
It is clear that $\cV$ and $\cW$ are closed subspaces of $L^2(\R^2)^6$ and orthogonal with respect to $\langle \cdot,\cdot\rangle_2$. Let  $\vp=\begin{pmatrix} \Phi\\ \tPhi \end{pmatrix}\in\cC_0^{\infty}(\R^2)^6$ and let $\alpha, \talpha\in H^{1}(\R^2)\cap \cC^{\infty}(\R^2)$ be the unique solutions to 
$$-\Delta \alpha+k^2 \alpha=-\big(\partial_{x_1} \Phi_1 +\partial_{x_2} \Phi_2 +k \tPhi_3\big)$$
and
$$-\Delta \talpha+k^2 \talpha=-\big(\partial_{x_1} \tPhi_1 +\partial_{x_2} \tPhi_2 -k \Phi_3\big)$$
respectively. Since $\alpha$ is the Bessel potential of the $\cC_0^\infty$-function $-(\partial_{x_1} \Phi_1 +\partial_{x_2} \Phi_2 +k \tPhi_3)$ we find that $\alpha\in W^{l,s}(\R^2)$ for any $l\in \N$ and $1<s<\infty$, cf. \cite[Chapter V §3]{stein}. Arguing similarly we see that the same also holds for $ \talpha$. Taking standard mollifiers and applying a cutoff argument,  we find $ \alpha_n,\talpha_n\in \cC_0^{\infty}(\R^2)$ such that $\alpha_n\to\alpha$ and $\talpha_n\to\talpha$ in $W^{l,s}(\R^2)$ for any $l\in\N$ and $1<s<\infty$. In particular this implies
$$\vp_\cW:=\cnabla \begin{pmatrix} \alpha\\ \talpha  \end{pmatrix}=\lim_{n\to\infty}
\cnabla \begin{pmatrix} \alpha_n\\ \talpha_n  \end{pmatrix} \in\cW$$
and $\vp_\cW\in W^{l,s}(\R^2)$ for all $l\in \N$ and $1<s<\infty$. Moreover, $\cW$ is clearly contained in the completion of   $\cC_0^{\infty}(\R^2)^6$ with respect to the norm $\|\cdot\|$. Since in particular $\vp_\cW\in H^1(\R^2)^6$ we obtain by integration by parts 
\begin{eqnarray*}
	\Big\langle \vp, \cnabla \begin{pmatrix} \beta\\ \tbeta  \end{pmatrix}\Big\rangle
	&=&-\int_{\R^2} (\partial_{x_1} \Phi_1 +\partial_{x_2} \Phi_2 +k \tPhi_3)\beta+
	(\partial_{x_1} \tPhi_1 +\partial_{x_2} \tPhi_2 -k \Phi_3)\tbeta\,dx\\
	&=&
	\int_{\R^2} (-\Delta \alpha+k^2 \alpha)\beta+
	(-\Delta \talpha+k^2 \talpha)\tbeta\,dx\\
	&=&
	\Big\langle \vp_\cW, \cnabla \begin{pmatrix} \beta\\ \tbeta  \end{pmatrix}\Big\rangle
\end{eqnarray*}
for every $\beta,\tbeta\in\cC_0^{\infty}(\R^2)$. Therefore $\vp_\cV:=\vp-\vp_\cW\in\cV$ and
we obtain the following Helmholtz-type decomposition
$$\vp =\vp_\cV+\vp_\cW\quad\hbox{with }\vp_\cV\in\cV\hbox{ and }\vp_\cW\in\cW$$
and moreover we have shown that $\vp\in X$. It remains to show that also $\cV$ is the completion of $\cC_0^\infty(\R^2)^6$ with respect to $\|\cdot\|$. To see this, let $v\in\cV$ and take $(\vp_n)\subset \cC_0^{\infty}(\R^2)^6$
such that $\vp_n\to v$ in $H^1(\R^2)^6$. Let us decompose $\vp_n=\vp_\cV^n+\vp_\cW^n\in\cV\oplus\cW$ and let $\cl_{L^2\cap L^p}\cV$ denote the closure of $\cV$ in $L^2(\R^2)^6\cap L^p(\R^2)^6$. Observe that $\big(\cl_{L^2\cap L^p}\cV\big)\cap \cW=\{0\}$, so that there is a continuous $L^2(\R^2)^6\cap L^p(\R^2)^6$-projection of $\big(\cl_{L^2\cap L^p}\cV\big)\oplus \cW$ onto $\cW$. Since $\vp_n\to v$ in $L^2(\R^2)^6\cap L^p(\R^2)^6$, we infer that $\vp_\cW^n\to 0$ in $\cW$. Similarly, arguing with the closure of $\cW$ in $H^1(\R^2)^6$ we get  $\vp_\cV^n\to v$ in $\cV$. Therefore
$$\|v-\vp_n\|=\|v-\vp_\cV^n\|+\|\vp^n_\cW\|\to 0$$ as $n\to\infty$ and we conclude that $\cV\oplus\cW$ is the completion of  $\cC_0^{\infty}(\R^2)^6$ with respect to the norm $\|\cdot\|$.
\end{proof}

Now we investigate the differential operator $L$ and its spectrum in the following two result.

\begin{Prop} The second-order differential operator $L:\dom(L) \subset L^2(\R^2)^6 \to L^2(\R^3)^6$ is elliptic and self-adjoint on the domain $\dom(L) = \{u\in L^2(\R^2)^6 \mbox{ s.t. } Lu\in L^2(\R^2)^6\}$. Its associated bilinear form $b_L: \dom(b_L)\times\dom(b_L)\to \R$ is given by 
$$
b_L(u,v)= \int_{\R^2} \cnabla\times u \cdot \cnabla\times v\,dx
$$
with $\dom(b_L)= \{u\in L^2(\R^2)^6 \mbox{ s.t. } \cnabla\times u\in L^2(\R^2)^6\}$.
\end{Prop}
\begin{proof} Let us show that $\dom(b_L)$ is closed in $L^2(\R^2)^6$. First note the symmetry property $\langle \cnabla u,v\rangle = \langle u,\cnabla v\rangle$ for all $u,v \in \dom(b_L)$. If $(u_n)_{n\in \N}\in \dom(b_L)$ is a sequence such $u_n\to u$ in $L^2(\R^2)^6$ and $\cnabla\times u_n\to \psi$ in $L^2(\R^2)^6$ for some $\psi\in L^2(\R^2)^6$ then the symmetry property implies  $\langle \cnabla\times u_n,\phi\rangle = \langle u_n,\cnabla\phi\rangle \to \langle u,\cnabla\phi\rangle$ for all $\phi\in \dom(b_L)$. Thus $\psi=\cnabla u$ and therefore $b_L$ is a closed symmetric bilinear form. It is therefore the associated bilinear form of a unique selfadjoint operator, cf. \cite[Theorem VIII.15]{reed_simon_vol1}. For all $u,v \in \cC_0^\infty(\R^2)^6$ we see that
$$
\langle Lu,\vp\rangle
	= \int_{\R^2} (\cnabla\times\cnabla\times u) \cdot \vp \,dx 
	= \int_{\R^2} (\cnabla\times u) \cdot (\cnabla\times \vp)\,dx 
	= b_L(u,\vp).
$$
Hence $b_L$ is the bilinear form of the operator $L: \dom(L) \to L^2(\R^2)^6$ with $\dom(L)=\{u\in L^2(\R^2)^6 \mbox{ s.t. } Lu\in L^2(\R^2)^6\}$.
\end{proof}

\begin{Th}\label{th:L}
	The operator $L:\dom(L) \subset L^2(\R^2)^6 \to L^2(\R^3)^6$ has spectrum $\sigma(L)={0}\cup [k^2,\infty)$, where $0$ is an eigenvalue of infinite multiplicity and $[k^2,\infty)$ consists of absolutely continuous spectrum. 
\end{Th}
\begin{proof}
	Let us consider the symbol $\hat L(\xi)$ which is a complex hermitian $6\times 6$ matrix for $\xi\in \R^2$. Let us denote by $\sigma(\hat L(\xi))$ the  matrix eigenvalues of $\hat L(\xi)$. We will verify $\sigma(L)=\overline{\bigcup_{\xi\in \R^2}\sigma(\hat L(\xi))}$ by two steps: 
	\begin{itemize}
		\item[(i)] $\lambda$ is a matrix eigenvalue of $\hat L(\xi)$ for some $\xi\in \R^2$ $\Rightarrow$ $\lambda\in \sigma(L)$
		\item[(ii)] $\dist(\lambda, \overline{\bigcup_{\xi\in \R^2}\sigma(\hat L(\xi))})>0$ $\Rightarrow$ $\lambda$ lies in the resolvent set of $L$
	\end{itemize}
	Since the spectrum of $L$ is closed, (i) and (ii) imply the claim. \\	
	\noindent
	{\em Proof of (i):} Let $a\in \R^6$ be a unit vector with $\hat L(\xi_0)a = \lambda a$. If $\rho\in \cC_0^\infty(\R^2)$ is a function with $\|\rho\|_{L^2(\R^2)}=1$ then let $\hat u_k(\xi) =  k \rho(k(\xi-\xi_0)) a$. One finds that $\|\hat u_k\|_{L^2(\R^2)^6}=1$ and 
	\begin{eqnarray*}
		\| (\hat L-\lambda\Id) \hat u_k\|_{L^2(\R^2)^6}^2 &=& \int_{\R^2} |\bigl(\hat L(\xi)-\hat L(\xi_0)\bigr) a|^2 k^2 \rho^2(k(\xi-\xi_0))\,d\xi \\
		&=& \int_{\R^2} |\bigl(\hat L(\xi_0 + k^{-1}\eta)-\hat L(\xi_0)\bigr) a|^2 \rho^2(\eta)\,d\eta
		\to 0 \mbox{ as } k \to \infty 
	\end{eqnarray*}
	by continuity of the symbol and dominated convergence. By Plancherel's theorem we have that $(L-\lambda)u_k\to 0$ as $k\to \infty$ in $L^2(\R^2)^6$ with $\|u_k\|_{L^2(\R^2)^6}=1$ so that $\lambda\in \sigma(L)$ by Weyl's criterion.	
	It turns out that the characteristic polynomial of $\hat L(\xi)$ is given by $\lambda^2(|\xi|^2+k^2-\lambda)^4$ and hence $\sigma(L)=\{0\}\cup[k^2,\infty)$.  Indeed,  taking the Fourier-transform we find the symbol of $L$, whose characteristic polynomial is given by 
	\begin{multline*} 
	\det(\hat L(\xi)-\lambda\Id) =  \\
	\det\begin{pmatrix} 
	\xi_2^2+k^2-\lambda & -\xi_1\xi_2 & 0  & 0 & 0 & ik\xi_1 \\
	-\xi_1\xi_2 & \xi_1^2+k^2-\lambda & 0 & 0 & 0 & ik\xi_2 \\
	0 & 0 & |\xi|^2-\lambda & ik\xi_1 & ik\xi_2 & 0 \\
	0 & 0 & -ik\xi_1 & \xi_2^2+k^2-\lambda & -\xi_1\xi_2 & 0  \\
	0 & 0 & -ik\xi_2 & -\xi_1\xi_2 & \xi_1^2+k^2-\lambda & 0 \\
	-ik\xi_1 & -ik\xi_2 & 0 & 0 & 0 & |\xi|^2-\lambda
	\end{pmatrix}.
	\end{multline*}
	Interchanging column 3 and 6 as well as line 3 and 6 we get a block structure which leads to  
	\begin{eqnarray*}
		\lefteqn{\det(\hat L(\xi)-\lambda\Id)} \\
		& = & \det\begin{pmatrix} 
			\xi_2^2+k^2-\lambda & -\xi_1\xi_2 & ik\xi_1 \\
			-\xi_1\xi_2 & \xi_1^2+k^2-\lambda & ik\xi_2 \\
			-ik\xi_1 & -ik\xi_2 & |\xi|^2-\lambda \end{pmatrix} 
		\det \begin{pmatrix}
			|\xi|^2-\lambda & ik\xi_1 & ik\xi_2 \\
			-ik\xi_1 & \xi_2^2+k^2-\lambda & -\xi_1\xi_2 \\
			-ik\xi_2 & -\xi_1\xi_2 & \xi_1^2+k^2-\lambda 
		\end{pmatrix} \\
		& = & \lambda^2 (|\xi|^2+k^2-\lambda)^4
	\end{eqnarray*}
	since both matrices have the same determinant $-\lambda(|\xi|^2+k^2-\lambda)^2$. Thus the matrix eigenvalues are given by $\sigma(\hat L(\xi))=\{0, |\xi|^2+k^2\}$. The zero eigenvalue of infinite multiplicity if generated by all vector fields $\cnabla\begin{pmatrix} \alpha \\ \talpha\end{pmatrix}$ with $\alpha,\talpha\in C_c^\infty(\R^2)$.
\end{proof}

Let $p_{\cV}:X\to\cV$, $p_{\cW}:X\to\cW$  denote the projections of $X$ onto $\cV$, $\cW$, respectively. Usually we just write $u=v+w\in\cV\oplus\cW$, where $v=p_{\cV}(u)\in \cV$ and $w=u-v=p_{\cW}(u)\in\cW$.  Observe that $\cV$ and $\cW$ are both subspaces of $\dom(b_L)$ and that for $v+w\in X=\cV\oplus\cW$ we have 
$$b_L(v+w,v+w)=b_L(v,v)=\|v\|^2.$$
For $u\in \dom(b_L)$ using the duality pairing between $\dom(b_L)'$ and $\dom(b_L)$ we can define $Lu$ by  setting $\langle Lu,\phi\rangle_{\dom(b_L)'\times \dom(b_L)}:= b_L(u,\phi)$ for all $\phi\in \dom(b_L)$. In this way $L: X\to \dom(b_L)'$ is well-defined and has the kernel $\cW$. Note that $L$ restricted to $\cV$ coincides with the vector Schr\"odinger operator $-\Delta+k^2$ acting diagonally on elements of $\cV$.

\medskip

We say that $u\in \begin{pmatrix} U \\ \wt U \end{pmatrix}\in X$ is a {\em weak solution} to \eqref{eq} provided that
\begin{equation}
b_L(u,\vp)-\int_{\R^2}V(x) \langle u,\vp\rangle \, dx- \int_{\R^2}\langle f(x,u),\vp\rangle\, dx=0
\end{equation}
for any $\vp\in\cC_0^{\infty}(\R^2)^6$. From now on we assume that (F1)--(F4) and (V) are satisfied.
Observe that for $u=v+w\in \cV\oplus\cW$ we get
\begin{equation*}
J(u):=\frac12\|v\|^2
-\frac12\int_{\R^2}V(x)|u|^2\, dx- \int_{\R^2}F(x,u) \, dx.
\end{equation*}

\begin{Cor}\label{Prop_energy}
$J:X\to\R$ is of class $\cC^1$ and $u\in X$ is a critical point of $J$ if and only if  $u$ is a weak solution to \eqref{eq}. Moreover, if \eqref{eq:application} holds and $E$ of the form \eqref{eq:travel_wave} is a travelling wave field with the profiles $U$ and $\tU$, where $u= \begin{pmatrix} U \\ \wt U \end{pmatrix}\in X$ is a critical point of $J$, then the total electromagnetic energy per unit interval on the $x_3$-axis is finite, i.e.
\begin{eqnarray*}
	\cL(t)&=&\frac12\int_{\R^2}\int_{a}^{a+1}\langle E,D\rangle +\langle B,H\rangle\,dx_3\, d(x_1,x_2)<\infty.
\end{eqnarray*}
\end{Cor}
\begin{proof}
The first statement is a consequence of Theorem \ref{eq:Helmholtz}.	According to the material laws
\begin{eqnarray*}
\langle E,D\rangle&=&	-\eps(x_1,x_2)\omega^2|E|^2+\chi(x_1,x_2, \langle |E|^2\rangle)|E|^2\\
&=&\big(-V(x_1,x_2)+\chi(x_1,x_2,\frac12 |u|^2)\big)\big(|U|^2\cos^2(kx_3+\omega t)+|\wt U|^2\sin^2(kx_3+\omega t)\big)
\end{eqnarray*} 
 and by the Faraday's law
\begin{eqnarray*}
\langle B,H\rangle &=& |B|^2 = \frac{1}{\omega^2}\big|\curlop \big(U(x_1,x_2) \sin(kx_3+\omega t)- \wt U(x_1,x_2)\cos(kx_3+\omega t)\big)\big|^2\\
&\leq & \frac{1}{\omega^2}|\cnabla\times u|^2.
\end{eqnarray*}
Therefore
$$
\begin{aligned}
\cL(t)\leq & \frac{1}{2\omega^2}b_L(u,u)+ \frac{1}{2}\int_{\R^2}\big(-V(x_1,x_2)+\chi(x_1,x_2,\frac12 |u|^2)\big) \\
		& \cdot \int_{a}^{a+1}\big(|U|^2\cos^2(kx_3+\omega t)+|\wt U|^2\sin^2(kx_3+\omega t)\big)\,dx_3d(x_1,x_2)\\
		\leq & \frac{1}{2}\bigl(1+\frac{1}{\omega^2}\bigr)b_L(u,u)<\infty.
\end{aligned}
$$
\end{proof}

\section{Abstract variational approach}\label{sec:criticaslpoitth}

Our approach is based on the critical point theory from \cite{BartschMederski,BartschMederskiJFA,MederskiSchinoSzulkin}, however sometimes we do not assume the monotonicity assumption (I8) given below and we present results with emphasis on where this condition is crucial.

Let
 $J:X\to\R$ be a functional of the form 
\begin{equation}\label{EqJ}
J(u) := \frac12\|u^+\|^2-I(u) \quad\text{for $u=u^++u^- \in X^+\oplus X^-$},
\end{equation}
such that  $X$ is a reflexive Banach space with the norm $\|\cdot\|$ and a topological direct sum decomposition $X=X^+\oplus X^-$, where $X^+$ is a Hilbert space with a scalar product $\langle \cdot , \cdot \rangle$. For $u\in X$ we denote by $u^+\in X^+$ and $u^- \in X^-$ the corresponding summands so that $u = u^++u^-$. We may assume $\langle u,u \rangle = \|u\|^2$ for any
$u\in X^+$ and $\|u\|^2 = \|u^+\|^2+\|u^- \|^2$. We define the topology $\cT$ on $X$ as the product of the norm topology in $X^+$ and the weak topology in $X^-$. Hence $u_n\cTto u$ if and only if $u_n^+\to u^+$ and $u_n^-\weakto u^-$. Let us define the set
\begin{equation}\label{eq:ConstraintM}
\cM := \{u\in X:\, J'(u)|_{X^-}=0\}=\{u\in X:\, I'(u)|_{X^-}=0\}.
\end{equation}
Clearly $\cM$ contains all critical points of $J$ and we assume the following conditions introduced in \cite{BartschMederski,BartschMederskiJFA}:
\begin{itemize}
	\item[(I1)] $I\in\cC^1(X,\R)$ and $I(u)\ge I(0)=0$ for any $u\in X$.
	\item[(I2)] $I$ is $\cT$-sequentially lower semicontinuous:
	$u_n\cTto u\quad\Longrightarrow\quad \liminf I(u_n)\ge I(u)$.
	\item[(I3)] If $u_n\cTto u$ and $I(u_n)\to I(u)$ then $u_n\to u$.
	\item[(I4)] $\|u^+\|+I(u)\to\infty$ as $\|u\|\to\infty$.
	\item[(I5)] If $u\in\cM$ then $I(u)<I(u+v)$ for every $v\in X^- \setminus\{0\}$.
	\item[(I6)] There exists $r>0$ such that $a:=\inf\limits_{u\in X^+,\|u\|=r} J(u)>0$.
	\item[(I7)]  $I(t_nu_n)/t_n^2\to\infty$ if $t_n\to\infty$ and $u_n^+\to u^+\ne 0$ as $n\to\infty$.
\end{itemize}
Observe that if $I$ is strictly convex, continuous and satisfies (I4), then (I2) and (I5) are clearly satisfied. Moreover, for any $u\in X^+$ we find  $m(u)\in\cM$ which is the unique global maximizer of $J|_{u+X^-}$. From now on we assume (I1)--(I7). Note that $m$ needs not be $\cC^1$, and $\cM$ needs not be a differentiable manifold since $I'$ is only continuous,  however from \cite[Proof of Theorem~4.4]{BartschMederskiJFA} we observe $m:X^+\to\cM$ is a homeomorphism and
$\wt{J}:=J\circ m: X^+\to \R$ is of class $\cC^1$.

We introduce the following min-max level 
\begin{equation}\label{cmX}
c:= \inf_{\ga\in\Ga}\sup_{t\in [0,1]} J(\ga(t)),
\end{equation}
where
\begin{equation}\label{def:Gamma}
\Ga := \big\{\ga\in\cC([0,1],\cM) \ : \ \ga(0)=0, \ \|\gamma(1)^+\|>r,\text{ and } J(\ga(1)^+)<0\big\}.
\end{equation}
If we look for ground state solutions with the least possible energy, we consider a Nehari-type constraint $\cN$ and  the following condition
\begin{itemize}
	\item[(I8)] $\frac{t^2-1}{2}I'(u)[u]+I(u) - I(tu+v)=\frac{t^2-1}{2}I'(u)[u] + tI'(u)[v] + I(u) - I(tu+v) \leq 0$\\ for every $u\in \cN$, $t\ge 0$, $v\in X^-$, where
	\begin{equation*}\label{eq:NehariDef}
	\cN := \{u\in X\setminus X^-: J'(u)|_{\R u\oplus X^-}=0\} = \{u\in\cM\setminus X^-: J'(u)[u]=0\} \subset\cM.
	\end{equation*}
\end{itemize}
The constraint $\cN$ is a generalized Nehari-type manifold inspired by a constrained considered in \cite{Pankov} for the Schr\"odinger problem.
We state the following result obtained in \cite[Theorem 3.3]{MederskiSchinoSzulkin} under assumptions (I1)--(I8).
\begin{Th}\label{ThLink1}
	Suppose $J$ of the form \eqref{EqJ} satisfies (I1)--(I7). Then $J$ has a sequence $(u_n)\subset \cM$ such that $J(u_n)\to c\geq a>0$ and $J'(u_n)(1+\|u_n^+\|)\to 0$ as $n\to\infty$.
If, in addition, (I8) holds, then $c=\inf_{\cN}J$.
\end{Th}
\begin{proof} For the reader's convenience, we shortly sketch the proof.
Recall  that  by (I5) and (I6),
$\wt J(u)\geq J(u)\geq a$ for $u\in X^+$ and $\|u\|=r$, and  $\wt J(tu)/t^2\to -\infty$ as $t\to\infty$. Thus $\wt J$ has the mountain pass geometry and similarly as in \cite[Theorem 4.4]{BartschMederskiJFA} we may define the mountain pass level
\begin{equation}\label{cm}
c_\cM :=\inf_{\sigma\in\wt\Ga}\sup_{t\in [0,1]} \wt J(\sigma(t))=\inf_{\ga\in\Ga}\sup_{t\in [0,1]} J(\ga(t))\geq a,
\end{equation}
where
\begin{eqnarray*}
\wt\Ga &:=& \{\sigma\in\cC([0,1],X^+) \ : \ \sigma(0)=0, \ \|\sigma(1)\|>r,\text{ and } \wt J(\sigma(1))<0\}.
\end{eqnarray*}
By the mountain pass theorem there exists a Cerami sequence $(v_n)$ for $\wt J$ at the level $c_\cM$. see \cite{Cerami,bbf}. Since $\wt J'(v_n)=J'(m(v_n))$ (see \cite{BartschMederskiJFA,MederskiSchinoSzulkin}), then setting $u_n:=m(v_n)$ we obtain
$$J'(u_n)(1+\|u_n^+\|)=\wt J'(v_n)(1+\|v_n\|)\to 0.$$
If (I8) holds, then by \cite[Theorem 3.3 (b)]{MederskiSchinoSzulkin}, $c=\inf_{\cM}J$.
\end{proof}

Now we present a multiplicity result. Recall that
for a topological group acting on $X$, \emph{the orbit of $u\in X$} is often denoted by $G\ast u$, i.e., 
$$G\ast u:=\{gu \ : \ g\in G\}.$$
A set $A\subset X$ is called \emph{$G$-invariant} if $gA\subset A$ for all $g\in G$. A functional $J: X\to\R$ is called \emph{$G$-invariant} and a map $T: X\to X^*$ is called \emph{$G$-equivariant} if $J(gu)=J(u)$ and $T(gu)=gT(u)$ for all $g\in G$, $u\in X$. In our application we use the action given by $\Z^2$-translations.

Assume that $G$ is a topological group such that
\begin{itemize}
	\item[(G)] $G$ acts on $X$ by isometries and discretely in the sense that for each $u\ne 0$, $(G*u)\setminus\{u\}$ is bounded away from $u$. Moreover, $J$ is $G$-invariant and $X^+, X^-$ are $G$-invariant.
\end{itemize}
Observe that $\cM$ is $G$-invariant and $m:X^+\to\cM$ is $G$-equivariant.

We shall use the notation 
\begin{gather*}
\wt J^\beta := \{u\in X^+ \ :\  \wt J(u)\le\beta\}, \quad \wt J_\alpha := \{u\in X^+\ : \  \wt J(u)\ge\alpha\}, \\
\wt J_\alpha^\beta := \wt J_\alpha\cap \wt J^\beta,\quad \cK:=\big\{u\in X^+\ :\ \wt J'(u)=0\big\}
\end{gather*}
and call $G*u$ a \emph{critical orbit} whenever $u\in \cK$. Moreover the following variant of the {\em Cerami condition} between the levels $\alpha, \beta\in\R$ has been introduced in \cite{MederskiSchinoSzulkin}: 
\begin{itemize}
	\item[$(M)_\alpha^\beta$]
	\begin{itemize}
		\item[(a)] Let $\alpha\le\beta$. There exists $M_\alpha^\beta$ such that $\limsup_{n\to\infty}\|u_n\|\le M_\alpha^{\beta}$ for every $(u_n)\subset X^+$ satisfying $\alpha\le\liminf_{n\to\infty}\wt J(u_n)\le\limsup_{n\to\infty}\wt J(u_n)\leq\beta$ and \linebreak $(1+\|u_n\|)\wt J'(u_n)\to 0$.
		\item[(b)] Suppose in addition that the number of critical orbits in $\wt J_\alpha^\beta$ is finite. Then there exists $m_\alpha^\beta>0$ such that if $(u_n), (v_n)$ are two sequences as above and $\|u_n-v_n\|<m_\alpha^\beta$ for all $n$ large, then  $\liminf_{n\to\infty}\|u_n-v_n\|=0$.
	\end{itemize}
\end{itemize}

Note that if $J$ is even, then $m$ is odd and $\cM$ is symmetric,  i.e. $\cM=-\cM$. Note also that $(M)_\alpha^\beta$ is a condition on $\wt J$. We slightly generalize \cite[Theorem 3.5 (b)]{MederskiSchinoSzulkin}.
\begin{Th}\label{Th:CrticMulti}
	Suppose $J$ of the form \eqref{EqJ} satisfies (I1)--(I7), $J$ is even, $\dim(X^+)=\infty$ and
	\begin{equation}\label{eq:CondMulti}
		\inf_{\cK}\wt J \geq 0.
	\end{equation}
	  If $(M)_0^{\beta}$ holds for every $\beta>0$, then $J$ has infinitely many distinct critical orbits.
\end{Th}

In fact, in \cite[Theorem 3.5 (b)]{MederskiSchinoSzulkin} condition (I8) has been  assumed instead of  \eqref{eq:CondMulti}. Note that if (I8) holds, then for every $v\in\cK\setminus\{0\}$ one has $m(v)\in\cN$, and $J(u)\geq J(ru^+/\|u^+\|)\geq a$ for any $u\in\cN$.
 Hence
$$\inf_{\cK\setminus \{0\}}\wt J\geq \inf_{\cN}J\geq a$$
so that all nontrivial critical points of $\wt J$ have energy at least $a>0$, in particular \eqref{eq:CondMulti} holds.  However, by the inspection of the proof of  \cite[Theorem 3.5 (b)]{MederskiSchinoSzulkin}, the following min-max values
\[
\beta_k :=  \inf_{A\in\Sigma,i^*(A)\ge k} \sup_{u\in A}\wt{\cJ}(u), \quad k=1,2,\ldots,
\]
are well-defined and finite,
where
 $\Sigma := \{A\subset X^+: A=-A \text{ and } A \text{ is compact}\}$,
and for $A\in\Sigma$, we define a variant of Benci's  pseudoindex \cite{bbf,MederskiSchinoSzulkin}.
\[
i^*(A) := \min_{h\in\cH} \gamma(h(A)\cap S(0,r))
\]
where $r$ is as in (I6), $S(0,r):=\{u\in X^+:\|u\|=r\}$, $\gamma$ is Krasnoselskii's genus and
$$\cH := \{h: X^+\to X^+ \text{ is a homeomorphism, }  h(-u)=-h(u) \text{ and } \wt{\cJ}(h(u))\le \wt{\cJ}(u) \text{ for all } u\}.
$$
We observe also that $\beta_k\geq a$. Hence, if we assume for contradiction that there is a finite number of distinct orbits $\{G\ast u: u\in \cK\}$, then using \eqref{eq:CondMulti} and arguing as in the proof of  \cite[Theorem 3.5 (b)]{MederskiSchinoSzulkin}, $\beta_k$ are critical values and
$a\le\beta_1<\beta_2<\ldots$, 
which is impossible by the assumption. Therefore Theorem \ref{Th:CrticMulti} holds true.

\medskip

Observe that if 
\begin{equation}\label{ARabstract}
I'(u)[u]\geq 2 I(u)\quad\hbox{for any }u\in X
\end{equation}
then $\wt J(u)=\wt J(u)-\frac12\wt J'(u)[u]\geq 0$ for any $u\in\cK$, so that \eqref{eq:CondMulti} holds. On the other hand, in applications where $J$ satisfies an Ambrosetti-Rabinowitz-type assumption e.g. (F6), then \eqref{ARabstract} is clearly satisfied.

\section{Application of the abstract variational approach}\label{sec:PS}

Recall that we assume (F1)--(F4) and (V). We assume, in addition, that $F$ is convex in $u\in\R^2$.
Let $X^+:=\cV$ and $X^-:=\cW$.
Let us define $I:X\to\R$ such that
$$I(u)=\frac12\int_{\R^2}V(x)|u|^2\, dx+\int_{\R^2}F(x,u) \, dx,$$
hence $J(v+w)=\frac{1}{2}\|v\|^2-I(v+w)$ for $v\in\cV$ and $w\in\cW$ so that $J$ is of the form \eqref{EqJ}. In this section we check that (I1)--(I7) are satisfied.

Let us define the following map
$$H^1(\R^2)^6\ni u\mapsto I_F(u):=\int_{\R^2}F(x,u)\,dx,$$
which is continuous and convex.

\begin{Lem}\label{prop:properties}
Conditions (I1)--(I4) are satisfied.
\end{Lem}
\begin{proof}
In view of (F1), (F2), (F3) and (V) we easily get (I1) and by the convexity assumption on $F$ also (I2). To see that (I3) holds let us suppose that $u_n\cTto u$ and $I(u_n)\to I(u)$. Then by the fact that $F\geq 0$ and weak lower semicontinuity 
\begin{equation}
\int_{\R^2}V(x)|u_n|^2\, dx\to \int_{\R^2}V(x)|u|^2\, dx, \quad I_F(u_n)\to I_F(u) \label{eq:conv1}
\end{equation}
and
passing to a subsequence $\big(V(x)^{1/2}u_n\big)\weakto \big(V(x)^{1/2}u\big)$ in $L^2(\R^2)^6$. Thus $u_n\to u$ in $L^2(\R^2)^6$ and $u_n(x)\to u(x)$ for a.e. $x\in\R^2$. Then, by Vitali's convergence theorem 
\begin{eqnarray*}
	I_F(u_n)-I_F(u_n-u)
	&=&\int_{\mathbb{R}^N}\int_0^1 -\frac{d}{ds}F(x,u_n-s u)\, ds\,dx\\\nonumber
	&=&\int_{\mathbb{R}^N}\int_0^1 f(x,u_n-su)u\,ds\, dx\\\nonumber
	&\rightarrow& \int_0^1 \int_{\mathbb{R}^N} f(x,u-su)u\,dx\,ds\\\nonumber
	&=&\int_{\mathbb{R}^N}\int_0^1 -\frac{d}{ds}F(x,u-su)\, ds\, dx\\
	&=&\int_{\mathbb{R}^N}F(x,u)\, dx=I_F(u)\nonumber
\end{eqnarray*}
as $n\to\infty$. From \eqref{eq:conv1} we infer that $I_F(u_n-u)\to 0$. Observe that, by (F4) and $F\geq0$, for any $\eps>0$ we find a constant $\tilde c_\eps>0$ such that 
\begin{equation}\label{eq:ineqFlp}
F(x,u)+\eps|u|^2\geq \tilde c_\eps|u|^p\quad\hbox{ for }x,u\in\R^2.
\end{equation}
Therefore 
$$\int_{\R^2}F(x,u_n-u)\, dx +|u_n-u|_2^2\geq \tilde c_1|u_n-u|^p_p$$
and we get $u_n\to u$ in $L^p(\R^2)^6$, which completes proof of (I3). Now note that if $I(v_n+w_n)$ is bounded with  $(v_n)\subset \cV$ and $w_n\subset \cW$ such that $v_n$ is bounded, then $w_n$ is bounded in $L^2(\R^2)^6$ and by \eqref{eq:ineqFlp}, $w_n$ is bounded in $L^p(\R^2)^6$. Thus (I4) holds.
\end{proof}

Let $v\in\cV$. Since $\cW\ni w\mapsto I(v+w)\in\R$ is strictly convex and coercive, $I(u)\ge I(0)=0$, then $I(v+\cdot)$ attains a unique global minimum at some $w(v)\in\cW$. Hence the set
\begin{equation*}
\cM := \{u\in X: J'(u)|_{\cW}=0\}=\{u\in X:\, I'(u)|_{\cW}=0\}
=\big\{u\in X:\, u=p_{\cV}(u)+w\big(p_{\cV}(u)\big)\big\},
\end{equation*}
obviously contains all critical points of $J$ and there holds
\begin{itemize}
	\item[(I5)] If $u\in\cM$ then $I(u)<I(u+w)$ for every $w\in \cW\setminus\{0\}$.
\end{itemize}
Since $I(v+w(v))\leq I(v)$ for $v\in\cV$, we see that $w$ maps bounded sets into bounded sets and in view of (I2) and (I3) one can easily show the continuity of $w$.

Now we show the following properties which imply the linking geometry in the spirit of Benci and Rabinowitz \cite{BenciRabinowitz}.
\begin{Lem} The following conditions are satisfied.\\
(I6) There exists $r>0$ such that $a:=\inf\limits_{v\in \cV,\|v\|=r} J(v)>0$.\\
(I7) $I(t_nu_n)/t_n^2\to\infty$ if $t_n\to\infty$ and $p_{\cV}(u_n)\to v\ne 0$ as $n\to\infty$.
\end{Lem}
\begin{proof}
Observe that (F2) and (F3) imply that for any $\eps>0$ there is $c_\eps>0$ such that
\begin{equation}\label{ceps}
|f(x,u)|\leq \eps |u|+c_\eps |u|^{p-1}\hbox{ and }F(x,u)\leq \eps |u|^2+c_\eps |u|^p\quad\hbox{for }x,u\in\R^2.
\end{equation}
Then, taking $0<\eps<k^2-\esssup V$, by (V) we get
\begin{eqnarray*}
J(v)&\geq &\frac12\int_{\R^2}|\nabla v|^2+(k^2-\esssup V-\eps)|v|^2\, dx-C_\eps |v|_p^p
\end{eqnarray*}
for $v\in\cV$,
and we obtain (I6) by the Sobolev embedding of $H^1(\R^2)^6$ into $L^p(\R^2)^6$ and by choosing $r>0$ sufficiently small. Now suppose that $I(t_nu_n)/t_n^2$ is bounded, $t_n\to\infty$ and $p_{\cV}(u_n)\to v\ne 0$ as $n\to\infty$. Note that if $u_n=v_n+w_n$ with $v_n=p_{\cV}(u_n)\in\cV$ and $w_n\in\cW$, then by \eqref{eq:ineqFlp}
\begin{eqnarray*}
I(t_nu_n)/t_n^2&\geq &\frac12\int_{\R^2}\big(\essinf V-\eps\big)|v_n+w_n|^2\,dx+t_n^{p-2}c_\eps|v_n+w_n|_p^p.
\end{eqnarray*}
Since $I(t_nu_n)/t_n^2$ is bounded and taking $0<\eps<\essinf V$ we obtain that $v_n+w_n\to 0$ in $L^p(\R^2)^6$. Since $v_n+w_n$ is also bounded in $L^2(\R^2)^6$ it has a weakly convergent subsequence and the weak $L^2$-limit of this subsequence has to coincide with the the strong $L^p$-limit. Thus, dropping subsequence indices, we get in total that $w_n\rightharpoonup -v\not =0$ in $\cW$ and obtain a contradiction.
\end{proof}

In view of Theorem \ref{ThLink1}, there is a sequence $u_n=v_n+w_n\in\cM$ such that $J'(u_n)\to c\geq a>0$ and $J'(u_n)(1+\|v_n\|)\to 0$ as $n\to\infty$.
We will show that $(u_n)$ is bounded if, in addition, (F5) or (F6) holds, see Lemma \ref{lem:bounded} in Section \ref{sec:proof}. 
The next section however is devoted to profile decompositions of bounded sequences in $\cM$, which will be important to show that there is a nontrivial weak limit point of the sequence $(u_n)$ up to translations, which is a critical point.

\section{Profile decompositions in $\cM$}

In addition to (F1)--(F4) and (V) we assume also that
$V=V_0$ or $V(x)>V_0$ for a.e. $x\in\R^3$. Let us define $I_0:X\to\R$ such that
$$I_0(u)=\frac12\int_{\R^2}V_0|u|^2\, dx+\int_{\R^2}F(x,u) \, dx.$$
Clearly $I_0$ satisfies analogous conditions (I1)--(I4) and for $v\in\cV$, $I_0(v+\cdot)$ attains a unique global minimum at some $w_0(v)\in\cW$. Thus, similarly as in \eqref{eq:ConstraintM}, we define 
$$\cM_0:=\{u\in X:\, I'_0(u)|_{\cW}=0\}.$$
Recall that profile decompositions of bounded sequences in $H^1(\R^N)$ have been obtained for instance by  Nawa \cite{Nawa}, Hmidi and Keraani \cite{HmidiKeraani}, and  by G\'erard \cite{Gerard} in $\dot{H}^s(\R^N)$. A similar result cannot be obtained in $X=\cV\oplus\cW$, since $\cW$ is not locally compactly embedded into $L^p(\R^2)$ for $p\geq 1$. However we prove the following decomposition result in the topological constraints $\cM,\cM_0\subset X$.

\begin{Th}\label{ThSplit}
	If $(u_n)$ is bounded in $\cM$, then, passing to a subsequence,
	there is $K\in \N\cup \{\infty\}$ and there are $\tu_0\in\cM$ and sequences $(\tu_i)_{i\geq 1}\subset \cM_0$, $(y_n^i)_{n\geq i}\subset \Z^2$ such that $y_n^0=0$, $|y_n^i-y_n^j|\to\infty$ as $n\to\infty$ for $i\neq j$, and the following conditions hold:\\
(a) If $K<\infty$, then $\tu_i\neq 0$ for $1\leq i\leq K$ and $\tu_i=0$ for $i>K$, if $K=\infty$, then $\tu_i\neq 0$ for  all $i\geq 1$.\\
(b) 
	$u_n(\cdot+y_n^i)\weakto \tu_i\hbox{ in }X$ for any $0\leq i < K+1$ \textnormal{(\footnote{If $K=\infty$, then $K+1=\infty$ as well.})}.\\
(c)
	$u_n(\cdot+y_n^i)\to\tu_i\hbox{ in }L_{loc}^{p}(\R^2)^6$ and a.e. in $\R^2$ for any $0\leq i < K+1$.\\
(d) There holds
\begin{eqnarray}\label{eq:Decomp1}
\lim_{n\to\infty}\|p_{\cV}(u_n)\|^2&\geq& \sum_{i=0}^{\infty} \|p_{\cV}(\tu_i)\|^2,\\\label{eq:Decomp2}
\lim_{n\to\infty}I(u_n)&\geq& I(\tu_0)+\sum_{i=1}^{\infty} I_0(\tu_i),\\\label{eq:Decomp3}
\lim_{n\to\infty}I_F(u_n)&=&\sum_{i=0}^{\infty} I_F(\tu_i),\\\label{eq:Decomp4}
&&\hspace{-3cm}\lim_{k\to\infty}\lim_{n\to\infty} I_F\big(u_n-\sum_{i=0}^k\tu_i(\cdot-y_n^i)\big)=0.
\end{eqnarray}
\end{Th}
\begin{proof} For a measurable set $A\subset\R^2$ we use the notation $\chi_A$ to denote the characteristic function of $A$. Let $u_n=v_n+w(v_n)\in\cM$, where $v_n=p_{\cV}(u_n)\in\cV$. Passing to a subsequence $\lim_{n\to\infty}\|v_n\|$, $\lim_{n\to\infty}I(u_n)$ and $\lim_{n\to\infty}I_F(u_n)$ exist and are finite.

{\em Part 1. Profile decomposition for $(v_n)$.}
  Take $r>\sqrt{2}$.
Since $(v_n)\subset H^1(\R^2)$ is bounded, we claim that, passing to a subsequence,
 there is $K\in \N\cup \{\infty\}$ and there is a sequence
$(\tv_i)_{i=0}^K\subset H^1(\R^2)$, for $0\leq i <K+1$  there are a sequence $(y_n^i)\subset \Z^2$ and positive numbers $(c_i)_{i=1}^{K}$ such that $y_n^0=0$ and for any $0\leq i<K+1$ one has
\begin{eqnarray}	\label{Eqxnxm1}
&&v_n(\cdot+y_n^i)\weakto\tv_i\hbox{ in }H^1(\R^2)\hbox{ and }v_n(\cdot+y_n^i)\chi_{B(0,n)}\to\tv_i\hbox{ in }L^{2}(\R^2)^6\cap L^{p}(\R^2)^6,\\	\label{Eqxnxm2}
&&\tv_i\neq 0\hbox{ if }i\geq 1,\\
\label{Eqxnxm}
&&|y_n^i-y_n^j|\geq n-2r\hbox{ for } j\neq i, 0\leq i,j< K+1\hbox{ and sufficiently large }n,\\\label{EqIntegralunSumci}
&&\int_{B(y_n^{i},r)}|\phi_n^{i-1}|^2\, dx \geq c_{i}\geq\frac{1}{2}\sup_{y\in\R^N}\int_{B(y,r)}|\phi_n^{i-1}|^2\, dx>0
 \hbox{ for }i\geq 1, \hbox{ and sufficiently large }n,
\end{eqnarray}
where 
$$\phi_n^i:=v_n-\sum_{j=0}^{i}\tv_j(\cdot -y_n^j)\quad\hbox{for }i\geq 0, n\geq 1.$$ 
Indeed,
passing to a subsequence we may assume that
\begin{eqnarray*}
	v_n &\weakto& \tv_0\quad \hbox{ in }H^1(\R^2)^6\\
	v_n\chi_{B(0,n)} &\to& \tv_0\quad \hbox{ in }L^{2}(\R^2)^6\cap L^p(\R^2)^6.
\end{eqnarray*}
The latter convergence follows from the fact that for any $n$, 
$H^1(B(0, n))$ is compactly embedded into $L^p(B(0, n))$ and we find sufficiently large $k_n$ such that 
$$|(v_{k_n}-\tv_0)\chi_{B(0,n)}|_{2} +|(v_{k_n}-\tv_0)\chi_{B(0,n)}|_{p}<\frac1n.$$
The subsequence $(u_{k_n})$ is then relabelled by $(u_n)$. 

Take $\phi_n^0:=v_n-\tv_0$ and
if
\begin{equation*}
\lim_{n\to\infty}\sup_{y\in\R^N}\int_{B(y,r)}|\phi_n^0|^2\, dx=0,
\end{equation*}
then we finish the proof of our claim with $K=0$.
Otherwise, passing to a subsequence, we find $(y_n^1)\subset\Z^2$ and a constant $c_1>0$ such that
\begin{equation}\label{eq:LemProofLions1}
\int_{B(y_n^{1},r)}|\phi_n^0|^2\, dx \geq c_{1}\geq \frac{1}{2}\sup_{y\in\R^N} \int_{B(y,r)}|\phi_n^0|^2\, dx>0.
\end{equation}
Note that $(y_n^1)$ is unbounded and we may assume that $|y_n^1|\geq n-r$.  Since $(v_n(\cdot+y_n^1))$ is bounded in $H^1(\R^2)$, we find $\tv_1\in H^1(\R^2)$ such that up to a subsequence
$v_n(\cdot+y_n^1)\weakto \tv_1$. In view of \eqref{eq:LemProofLions1},  we get $\tv_1\neq 0$, and again we may assume that $v_n(\cdot+y_n^1)\chi_{B(0,n)}\to \tv_1$ in $L^{2}(\R^2)^6\cap L^p(\R^2)^6$.  
We set  $\phi_n^1:=\phi_n^0-\tv_1(\cdot -y_n^1)=v_n-\tv_0-\tv_1(\cdot -y_n^1)$ and observe that if 
\begin{equation*}
\lim_{n\to\infty}\sup_{y\in\R^2}\int_{B(y,r)}|\phi_n^1|^2\, dx=0,
\end{equation*}
then we finish the proof of our claim with $K=1$. Otherwise, passing to a subsequence, we find $(y_n^2)\subset\Z^2$ and a constant $c_2>0$ such that
\begin{equation}\label{eq:LemProofLions2}
\int_{B(y_n^{2},r)}|\phi_n^1|^2\, dx \geq c_{2}\geq \frac{1}{2}\sup_{y\in\R^2} \int_{B(y,r)}|\phi_n^1|^2\, dx>0
\end{equation}
and $|y_n^2|\geq n-r$. Moreover $|y_n^2-y_n^1|\geq n-2r$. Otherwise $B(y_n^2,r)\subset B(y_n^1,n)$ and the convergence $\tv_0\chi_{B(y_n^2,r)}\to 0$ and $v_n(\cdot+y_n^1)\chi_{B(0,n)}\to \tv_1$ in $L^2(\R^2)^6$ contradict \eqref{eq:LemProofLions2}.
Then we find $\tv_2\neq 0$ such that passing to a subsequence 
$$\phi_n^1(\cdot+y_n^2),\;v_n(\cdot +y_n^2)\weakto \tv_2\text{ in }H^1(\R^2)\hbox{ and }v_n(\cdot+y_n^2)\chi_{B(0,n)}\to \tv_2\hbox{ in } L^{2}(\R^2)^6\cap L^p(\R^2)^6.$$ Again, if 
\begin{equation*}
\lim_{n\to\infty}\sup_{y\in\R^2}\int_{B(y,r)}|\phi_n^2|^2\, dx=0,
\end{equation*}
where $v_n^2:=\phi_n^1-\tv_2(\cdot-y_n^2)$,
then we finish proof with $K=2$.
Continuing the above procedure we finally find $K\in \N\cup \{\infty\}$ such that for $0\leq i<K+1$, \eqref{Eqxnxm1}--\eqref{EqIntegralunSumci} hold. In view of the Brezis-Lieb Lemma \cite{BrezisLieb},   
$$\limsup_{n\to\infty} |v_n|_p=|\tv_0|_p+\limsup_{n\to\infty} |\phi_n^0|_p=|\tv_0|_p+|\tv_1|_p+\limsup_{n\to\infty} |\phi_n^1|_p=
|\tv_0|_p+...+|\tv_i|_p+\limsup_{n\to\infty} |\phi_n^i|_p$$
for $i\geq 0$.
If there is $i\geq 0$ such that 
$$\lim_{n\to\infty}\sup_{y\in\R^2}\int_{B(y,r)}|\phi_n^i|^2\, dx=0,$$
then $K=i$ and setting $\tv_k=0$ for $k>i$ we get by Lion's Lemma \cite[Lemma 1.21]{Willem}
$$
\lim_{n\to\infty} |\phi_{n}^i|_{p}=0.
$$
Since $\phi_n^k=\phi_n^i$ for $k\geq i$, we find in particular
\begin{equation}\label{eq:phink}
\lim_{k\to\infty}\limsup_{n\to\infty} |\phi_{n}^k|_{p}=0. 
\end{equation}
In the case $K=\infty$ we want to show that \eqref{eq:phink} is still satisfied. We argue as in \cite[Proof of Theorem 1.4]{MederskiBL}.
Suppose, for a contradiction, that $\limsup_{k\to\infty}\limsup_{n\to\infty} |\phi_{n}^k|_{p}>0$. 
Then we find $\delta>0$ and increasing sequences $(i_k), (n_k)\subset \N$ such that
$$\int_{\R^2}|\phi_{n_k}^{i_k}|^p\, dx>\delta$$
and 
\begin{equation} \label{eq:added_by_W}
\sup_{y\in\R^2}\int_{B(y,r)}|\phi_{n_k}^{i_k}|^2\,dx\leq \limsup_{n\to\infty}\Big(\sup_{y\in\R^2}\int_{B(y,r)}|\phi_n^{i_k}|^2\,dx\Big)+\frac{1}{i_k}.
\end{equation}
Note that by \eqref{EqIntegralunSumci} we have
\begin{eqnarray*}
	c_{k+1}&\leq& \int_{B(y_n^{k+1},r)}
	|\phi_n^k|^2\, dx\\
	&\leq&2\int_{B(y_n^{k+1},r)}|\phi_n^i|^2\, dx+2\int_{B(y_n^{k+1},r)}\Big|\sum_{j=i+1}^{k}\tv_j(\cdot -y_n^j)\Big|^2\, dx\\
	&\leq& 4 c_{i+1} + 2(k-i)\sum_{j=i+1}^{k}\int_{B(y_n^{k+1}-y_n^j,r)}|\tv_j|^2\, dx
\end{eqnarray*}
for any $0\leq  i<k$.
Taking into account \eqref{Eqxnxm} and letting $n\to\infty$ we get $c_{k+1}\leq 4 c_{i+1}$. Take $k\geq 1$ and   $n>4r$. Again by \eqref{EqIntegralunSumci} and \eqref{Eqxnxm}  we obtain
\begin{equation*}
\begin{aligned}
\frac{1}{16}\sup_{y\in\R^N}\int_{B(y,r)}|\phi_n^k|^2\,dx
&\leq \frac{1}{8} c_{k+1}\leq \frac{1}{2k} \sum_{i=0}^{k-1} c_{i+1}\leq\frac{1}{2k} \sum_{i=0}^{k-1}
\int_{B(y_n^{i+1},r)}|\phi_n^i|^2\, dx\\
&\leq\frac{1}{k} \sum_{i=0}^{k-1}
\int_{B(y_n^{i+1},r)}|v_n|^2 + \Big|\sum_{j=0}^{i}\tv_j(\cdot -y_n^j)\Big|^2\, dx\\
&= \frac{1}{k}
\int_{\bigcup_{i=0}^{k-1}B(y_n^{i+1},r)}\left(|v_n|^2\,dx +\frac1k\int_{\R^N} \Big|\sum_{i=0}^{k-1}\sum_{j=0}^{i}\tv_j(\cdot -y_n^j)\chi_{B(y_n^{i+1},r)}\Big|^2\,\right) dx\\
&\leq \frac{1}{k}
|v_n|_2^2+\frac1k\Big|\sum_{i=0}^{k-1}\sum_{j=0}^{i}\tv_j(\cdot -y_n^j)\chi_{B(y_n^{i+1},r)}\Big|^2_2.
\end{aligned}
\end{equation*}
Observe that by \eqref{Eqxnxm} and since $n>4r$ we have
$$B(y_n^{i+1}-y_n^j,r)\subset \R^2\setminus B(0,n-3r)\hbox{ for }0\leq j\leq  i<k$$
and
\begin{eqnarray*}
	\Big|\sum_{i=0}^{k-1}\sum_{j=0}^{i}\bar{v}_j(\cdot -y_n^j)\chi_{B(y_n^{i+1},r)}\Big|_{2}&\leq& \nonumber
	\sum_{i=0}^{k-1}\sum_{j=0}^{i}\big|\bar{v}_j\chi_{B(y_n^{i+1}-y_n^j,r)}\big|_{2}\leq \sum_{i=0}^{k-1}\sum_{j=0}^{i}\big|\bar{v}_j\chi_{\R^2\setminus B(0,n-3r)}\big|_{2}\\
	&\leq& k\sum_{j=0}^{k-1}\big|\bar{v}_j\chi_{\R^2\setminus B(0,n-3r)}\big|_{2}\to 0
\end{eqnarray*}
as $n\to\infty$. Hence
\begin{equation}\label{eq:ineqGer1}
\limsup_{n\to\infty}\Big(\sup_{y\in\R^2}\int_{B(y,r)}|\phi_n^k|^2\,dx\Big)\leq  \frac{32}{k}
\limsup_{n\to\infty}|v_n|_2^2.
\end{equation}
Therefore we obtain from \eqref{eq:added_by_W}
$$\lim_{k\to\infty} \Big(\sup_{y\in\R^N}\int_{B(y,r)}|\phi_{n_k}^{i_k}|^2\,dx\Big)=0,$$
and in view of \cite[Lemma 1.21]{Willem} we obtain that $\phi_{n_k} ^{i_k}\to 0$ in $L^p(\R^2)^6$ as $k\to\infty$, which is a contradiction. Therefore passing to a subsequence \eqref{eq:phink} holds.	

{\em Part 2. Profile decomposition for $(u_n)$.}
	Note that $(w(v_n))_{n\in\N}$ and $(w_0(v_n))_{n\in\N}$ are bounded and we may assume
	\begin{eqnarray}\label{Eqweakw}
	&&w(v_n)(\cdot+y_n^i)\rightharpoonup\tw_i\quad\hbox{in }L^{2}(\R^2)^6\cap L^{p}(\R^2)^6\hbox{ for } i\geq 0
	\end{eqnarray}
	for some $\tilde w_i\in\cW$, $i\geq 0$. Let us define $\tu_0:=\tv_0+w(\tv_0)$ and $\tu_i:=\tv_i+w_0(\tv_i)$ if $i\geq 1$.

{\em Part 2. Claim 1.}
 There holds
	\begin{equation}\label{EqSeries}
	\sum_{i=1}^{\infty}
	I_F(\tu_i)\leq \sum_{i=1}^{\infty}
	I_0(\tu_i)<+\infty.
	\end{equation}
	Indeed, observe that for $i\geq 1$
	$$(v_n+w(v_n))(\cdot+y_n^i)\chi_{B(0,\frac{n-2r}{2})}\weakto \tv_i+\tw_i\hbox{ weakly in } X $$
	so that the weak lower semicontinuity of $I_0$ implies
	\begin{eqnarray*}
		\sum_{i=1}^k I_0(\tu_i)&\leq&
		\sum_{i=1}^k I_0(\tv_i+\tw_i)\leq \sum_{i=1}^k \liminf_{n\to\infty} I_0\big((v_n+w(v_n))(\cdot+y_n^i)\chi_{B(0,\frac{n-2r}{2})}\big)\\
		&\leq& \liminf_{n\to\infty}\sum_{i=1}^k I_0\big((v_n+w(v_n))\chi_{B(y_n^i,\frac{n-2r}{2})}\big)
		\leq
		\liminf_{n\to\infty}I_0(v_n+w(v_n))=\liminf_{n\to\infty}I_0(u_n)
	\end{eqnarray*}
	for any $k\geq 0$.  Since $(u_n)$ and hence $(I_0(u_n))$ is bounded,  (\ref{EqSeries}) holds.

	{\it Part 2. Claim 2.}  Up to a subsequence, there holds
	\begin{eqnarray}\label{Ineq15}
	&&\sup_{k\geq 0}\limsup_{n\to\infty}I_F\Big(
	\sum_{i=0}^k\tu_i(\cdot-y_n^i)\Big)\leq 
	\sum_{i=0}^{\infty} I_F(\tu_i).
	\end{eqnarray}
	Observe that
	for given $k\geq 0$
	\begin{equation}
	\begin{aligned}\label{Eqnxi}
	I_F\Big(
	\sum_{i=0}^k\tu_i(\cdot-y_n^i)\Big)
	&= 
	I_F\Big(\sum_{i=0}^k\tu_i(\cdot-y_n^i)
	\chi_{\bigcup_{j=0}^kB(y_n^j,\frac{n-2r}{2})}\Big)\\
	&\hspace{0.4cm}+I_F\Big(\sum_{i=0}^k\tu_i(\cdot-y_n^i)\chi_{\R^2\setminus \bigcup_{j=0}^k B(y_n^j,\frac{n-2r}{2})}\Big).
	\end{aligned}
	\end{equation}
Concerning the first sum on the right hand side of \eqref{Eqnxi} note that, for given $0\leq j\leq k$ 
\begin{eqnarray}\label{eq:convClaim2}
	\Big|\sum_{0\leq i\leq k,i\neq j}\tu_i(\cdot-(y_n^i-y_n^j))\chi_{B(0,\frac{n-2r}{2})}\Big|_{p}
	&\leq&
	\sum_{0\leq i\leq k,i\neq j}\Big|\tu_i\chi_{B(y_n^j-y_n^i,\frac{n-2r}{2})}\Big|_{p}\\
	&\to& 0\nonumber
\end{eqnarray}
as $n\to\infty$,
since 
$$B\Big(y_n^i-y_n^j,\frac{n-2r}{2}\Big)\subset \R^2\setminus B\Big(0,\frac{n-2r}{2}\Big)$$
for $i\neq j$. The convergence in \eqref{eq:convClaim2} also holds for $p=2$. Taking into account the uniform continuity of $I_F$ on bounded sets we obtain
	\begin{eqnarray*}
\lefteqn{I_F\Big(\sum_{i=0}^k\tu_i(\cdot-y_n^i)
\chi_{\bigcup_{j=0}^kB(y_n^j,\frac{n-2r}{2})}\Big)} \\
&=&
I_F\Big(\sum_{i=0}^k\sum_{j=0}^k\tu_i(\cdot-y_n^i)
\chi_{B(y_n^j,\frac{n-2r}{2})}\Big)\\
&=&I_F\Big(\sum_{i=0}^k \tu_i(\cdot-y_n^i)
\chi_{B(y_n^i,\frac{n-2r}{2})}+\sum_{0\leq i\neq j\leq k}\tu_i(\cdot-y_n^i)
\chi_{B(y_n^j,\frac{n-2r}{2})}\Big)\\
&\to& \sum_{i=0}^k  I_F(\tu_i)
\end{eqnarray*}
as $n\to\infty$.
Finally, concerning the second sum on the right hand side of \eqref{Eqnxi} observe that for any $i\geq 0$
$$\big|\tu_i(\cdot-y_n^i)\chi_{\R^2\setminus \bigcup_{j=0}^k B(y_n^j,\frac{n-2r}{2})}\big|_p\leq 
\big|\tu_i(\cdot-y_n^i)\chi_{\R^2\setminus B(y_n^i,\frac{n-2r}{2})}\big|_p=
\big|\tu_i\chi_{\R^2\setminus B(0,\frac{n-2r}{2})}\big|_p\to 0$$
as $n\to \infty$. The same convergence also holds for $p=2$. Taking the $\limsup_{n\to \infty}$ and the $\sup_{k\geq 0}$ in \eqref{Eqnxi} we conclude the proof of \eqref{Ineq15}.\\
{\it Part 2. Claim 3.}		
For any $k\geq 0$ 
		\begin{eqnarray}\label{eq:clam2.3}
			&&\liminf_{n\to\infty}\int_{\R^2}|v_n+w(v_n)|^2\,dx-\sum_{i=0}^k\int_{\R^2}|\tv_i+\tw_i|^2\,dx	\\\nonumber
			&&\hspace{20mm}\geq \liminf_{n\to\infty} \int_{\R^2}|v_n+w(\tv_0)+\sum_{i=1}^kw_0(\tv_i)(\cdot-y_n^i)|^2\,dx-\sum_{i=0}^k\int_{\R^2}|\tu_i|^2\,dx\nonumber
		\end{eqnarray}
		as $n\to\infty$.
		 Indeed, observe that
	\begin{eqnarray*}
	&&\int_{\R^2}|v_n+w(v_n)|^2\,dx-\sum_{i=0}^k\int_{\R^2}|\tv_i+\tw_i|^2\,dx- \int_{\R^2}|v_n+w(\tv_0)-\sum_{i=1}^kw_0(\tv_i)(\cdot-y_n^i)|^2\,dx\\
	&&\hspace{5mm}+\sum_{i=0}^k\int_{\R^2}|\tu_i|^2\,dx\\
	&&=\int_{\R^2}|w(v_n)|^2\,dx-\sum_{i=0}^k\int_{\R^2}|\tw_i|^2\,dx	-\int_{\R^2}|w(\tv_0)-\sum_{i=1}^kw_0(\tv_i)(\cdot-y_n^i)|^2\,dx\\ &&\hspace{5mm}+\int_{\R^2}|w(\tv_0)|^2\,dx+\sum_{i=1}^k\int_{\R^2}|w_0(\tv_i)|^2\,dx\\
	&&\geq \sum_{i=0}^k\int_{B(y_n^i,\frac{n-2r}{2})}|w(v_n)|^2\,dx-\sum_{i=0}^k\int_{\R^2}|\tw_i|^2\,dx\\
	&&\hspace{5mm}+2\sum_{i=1}^k\int_{\R^2}\langle w(\tv_0), w_0(\tv_i)(\cdot-y_n^i)\rangle\,dx-2\sum_{1\leq i< j\leq k}\int_{\R^2}\langle w_0(\tv_i)(\cdot-y_n^i), w_0(\tv_j)(\cdot-y_n^j)\rangle\,dx.
\end{eqnarray*}		
Note that for any $i\geq 1$ and $R>0$
\begin{eqnarray*}
\lefteqn{\int_{\R^2}|\langle w(\tv_0), w_0(\tv_i)(\cdot-y_n^i)\rangle|\,dx} \\
&\leq &
\int_{B(0,R)}|\langle w(\tv_0), w_0(\tv_i)(\cdot-y_n^i)\rangle|\,dx
+|w(\tv_0)\chi_{\R^2\setminus B(0,R)}|_2|w_0(\tv_i)|_2\\
&\leq &
|w(\tv_0)|_2 |w_0(\tv_i)\chi_{B(-y_n^i,R)}|_2
+|w(\tv_0)\chi_{\R^2\setminus B(0,R)}|_2|w_0(\tv_i)|_2\\
&=& |w(\tv_0)\chi_{\R^2\setminus B(0,R)}|_2|w_0(\tv_i)|_2+o(1) 
\end{eqnarray*}
as $n\to\infty$. Letting $R\to\infty$ we obtain $\int_{\R^2}|\langle w(\tv_0), w_0(\tv_i)(\cdot-y_n^i)\rangle|\,dx=o(1)$
and we conclude the claim, since $w(v_n)(\cdot+y_n^i)\chi_{B(0,\frac{n-2r}{2})}\weakto \tw_i$ in $L^2(\R^2)^6$.

{\it Part 2. Claim 4.} 	For any $k\geq 0$
\begin{equation}\label{eq:claim2.4}
\lim_{n\to\infty}\int_{\R^2}(V(x)-V_0)
\Big|\sum_{i=0}^k\tu_i(\cdot-y_n^i)\Big|^2\,dx=\int_{\R^2}(V(x)-V_0)
|\tu_0|^2\,dx.
\end{equation}
Note that if  $i\geq j$ and $i\geq 1$, then for any $R>0$ we have
\begin{eqnarray*}
	\lefteqn{\int_{\R^2}|V(x)-V_0|\big|\langle \tu_i(\cdot-y_n^i), \tu_j(\cdot-y_n^j)\rangle\big|\,dx}\\
	&\leq &
	\int_{B(-y_n^i,R)}|V(x+y_n^i)-V_0|\big|\langle \tu_i, \tu_j(\cdot-(y_n^j-y_n^i))\rangle\big|\,dx\\
	&&+|(V(x)-V_0)\chi_{\R^2\setminus B(0,R)}|_{\frac{p}{p-2}}|\tu_i|_p
	|\tu_j|_p\\
	&=& |(V(x)-V_0)\chi_{\R^2\setminus B(0,R)}|_{\frac{p}{p-2}}|\tu_i|_p
	|\tu_j|_p+o(1) 
\end{eqnarray*}
as $n\to\infty$. Letting $R\to\infty$ we conclude the claim.
		
	{\it Part 2. Claim 5.} Up to a subsequence we have
	\begin{eqnarray}\label{Claim4_1}
	&& w(v_n)\chi_{B(0,\frac{n-2r}{2})}\to w(\tv_0)\hbox{ in }L^{2}(\R^2)^6\cap L^{p}(\R^2)^6,\\\label{Claim4_2}
	&&w(v_n)(\cdot+y_n^i)\chi_{B(0,\frac{n-2r}{2})}\to w_0(\tv_i)\hbox{ in }L^{2}(\R^2)^6\cap L^{p}(\R^2)^6,
	\end{eqnarray}
	as $n\to\infty$.\\
	Let us denote $\nu_n^k:=v_n+w(\tv_0)+\sum_{i=1}^kw_0(\tv_i)(\cdot-y_n^i)$. 	Since the map 
	$$L^p(\R^2)^6\ni u\mapsto \frac12\int_{\R^2}(V(x)-V_0)|u|^2\,dx+I_F(u)$$
	is uniformly continuous on bounded sets, then in view of \eqref{eq:phink} for every $\eps>0$ there is $k_0\geq 0$ such that for any $k\geq k_0$ there is $n_0=n_0(k)$ such that for $n\geq n_0$ one obtains 
\begin{eqnarray}\label{eq:das}
	&&\eps+\frac{1}{2}\int_{\R^2}(V(x)-V_0)
	\Big|\sum_{i=0}^k\tu_i(\cdot-y_n^i)\Big|^2\,dx+
	I_F\Big(\sum_{i=0}^k\tu_i(\cdot-y_n^i)\Big)+\frac{1}{2}\int_{\R^2}V_0
	|\nu_n^k|^2\,dx\\\nonumber
	&&\geq \frac{1}{2}\int_{\R^2}(V(x)-V_0)
	|\nu_n^k|^2\,dx+
	I_F(\nu_n^k)+\frac{1}{2}\int_{\R^2}V_0
	|\nu_n^k|^2\,dx= I(\nu_n^k)\\\nonumber
	&&\geq
	I(u_n)=\frac{1}{2}\int_{\R^2}(V(x)-V_0)
	|u_n|^2\,dx+
	I_F(u_n)+\frac{1}{2}\int_{\R^2}V_0
	|u_n|^2\,dx.\nonumber
\end{eqnarray}
Moreover
\begin{eqnarray*}
	\liminf_{n\to\infty}\int_{\R^2}(V(x)-V_0)|u_n|^2\,dx
	&\geq&  \int_{\R^2}(V(x)-V_0)|\tv_0+\tw_0|^2\,dx,
\end{eqnarray*}
and
	\begin{eqnarray}\label{eq:claimI_F}
	\liminf_{n\to\infty}I_F(u_n)
	&\geq&\liminf_{n\to\infty}\sum_{i=0}^k I_F\big((v_n+w(v_n))\chi_{B(y_n^i,\frac{n-2r}{2})}\big)\\
	&\geq& \sum_{i=0}^k \liminf_{n\to\infty}I_F\big((v_n+w(v_n))\chi_{B(y_n^i,\frac{n-2r}{2})}\big)
	\geq
	\sum_{i=0}^k I_F(\tv_i+\tw_i)\nonumber
\end{eqnarray}
for any $k\geq 0$.
Then
\begin{eqnarray*}
\liminf_{n\to\infty}I(u_n)&\geq& 
\frac12\int_{\R^2}(V(x)-V_0)|\tv_0+\tw_0|^2\,dx+
\sum_{i=0}^k I_F(\tv_i+\tw_i)\\
&&+\liminf_{n\to\infty}\frac12\int_{\R^2}V_0|v_n+w(v_n)|^2\,dx\\
&=& I(\tv_0+\tw_0)+\sum_{i=1}^k I_0(\tv_i+\tw_i)-
\sum_{i=0}^k \frac12\int_{\R^2}V_0|\tv_i+\tw_i|^2\,dx\\
&&+\liminf_{n\to\infty}\frac12\int_{\R^2}V_0|v_n+w(v_n)|^2\,dx
\\
&\geq & I(\tu_0)+\sum_{i=1}^k I_0(\tu_i)-
\sum_{i=0}^k\frac12\int_{\R^2}V_0|\tv_i+\tw_i|^2\,dx\\
&&+\liminf_{n\to\infty}\frac12\int_{\R^2}V_0|v_n+w(v_n)|^2\,dx\\ 
	&\geq&	 \frac{1}{2}\int_{\R^2}(V(x)-V_0)|\tu_0|^2\,dx+\sum_{i=0}^k I_F(\tu_i)+\liminf_{n\to\infty}\frac{1}{2}\int_{\R^2}V_0
	|\nu_n^k|^2\,dx,
\end{eqnarray*}
where the last inequality follows from \eqref{eq:clam2.3}.

On the other hand, taking into account \eqref{eq:das} and then \eqref{eq:claim2.4}, \eqref{Ineq15} we get
\begin{eqnarray*}
	\liminf_{n\to\infty}I(u_n)&\leq&  \eps+  \frac{1}{2}\int_{\R^2}(V(x)-V_0)|\tu_0|^2\,dx+\sum_{i=0}^{\infty} I_F(\tu_i)+\sup_{k\geq 0}\liminf_{n\to\infty}\frac{1}{2}\int_{\R^2}V_0
	|\nu_n^k|^2\,dx,
\end{eqnarray*}
hence
\begin{equation}\label{eq:equalityClaim}
	\liminf_{n\to\infty}I(u_n)= \frac{1}{2}\int_{\R^2}(V(x)-V_0)|\tu_0|^2\,dx+\sum_{i=0}^{\infty} I_F(\tu_i)+\sup_{k\geq 0}\liminf_{n\to\infty}\frac{1}{2}\int_{\R^2}V_0
	|\nu_n^k|^2\,dx.
\end{equation}
Therefore we get equalities in the above considerations, in particular $\tu_i=\tv_i+\tw_i$ for $i\geq 0$ and in \eqref{eq:clam2.3} equality holds for $k=\infty$, so that $|w(v_n)(\cdot+y_n^i)\chi_{B(0,\frac{n-2r}{2})}|_2\to |\tw_i|_2$ passing to a subsequence, and then $w(v_n)(\cdot+y_n^i)\chi_{B(0,\frac{n-2r}{2})}\to \tw_i$ in $L^2(\R^2)^6$ for $i\geq 0$. Moreover from equality in \eqref{eq:claimI_F} we get
	\begin{eqnarray*}
		&&\liminf_{n\to\infty}I_F\big(u_n\chi_{B(0,\frac{n-2r}{2})})=I_F(\tv_0+w(\tu_0)),\\
		&&\liminf_{n\to\infty} I_F\big(u_n(\cdot+y_n^i)\chi_{B(0,\frac{n-2r}{2})})=I_F(\tv_i+w_0(\tu_i))
	\end{eqnarray*}
for $i\geq 1$.
	Taking into account \eqref{Eqxnxm1} and arguing as in the proof of (I3) from Lemma~\ref{prop:properties}, passing to a subsequence, we obtain
	\begin{eqnarray*}
		&&u_n\chi_{B(0,\frac{n-2r}{2})}\to \tv_0+w(\tv_0)\hbox{ in }L^{p}(\R^2)^6,\\
		&&u_n(\cdot+y_n^i)\chi_{B(0,\frac{n-2r}{2})}\to \tv_i+ w_0(\tv_i)\hbox{ in }L^{p}(\R^2)^6. 
	\end{eqnarray*}
	Therefore  \eqref{Claim4_1} and \eqref{Claim4_2} are satisfied. 
	
{\em Part 3. Proof of (d).}	
Now we complete the proof.	Note that we have already proved (a)--(c). Since $v_n(\cdot+y_n^i)\chi_{B(0,\frac{n-2r}{2})}\weakto\tv_i$ and arguing similarly as in \eqref{eq:claimI_F} we obtain 
\eqref{eq:Decomp1}. Taking into account \eqref{eq:equalityClaim} we find \eqref{eq:Decomp2} and  using the fact that the equality holds in  \eqref{eq:claimI_F} for $k=\infty$ we conclude \eqref{eq:Decomp3}.
 
 In order to show \eqref{eq:Decomp4}, note that $u_n(\cdot+y_n^i)\to \tu_i$ for a.e. in $\R^2$ and for every $0\leq i<K+1$. Starting with $i=0$ and arguing similarly as in proof of Lemma \ref{prop:properties} we get
	$$\lim_{n\to\infty}\big(I_F(u_n)-I_F(u_n-\tu_0)\big)=I_F(\tu_0).
	$$
	Since we know from the beginning of the proof that $\lim_{n\to\infty} I_F(u_n)$ exists and is finite, we therefore can rearrange and find
	\begin{equation} \label{rel1}
	\lim_{n\to\infty}I_F(u_n)=I_F(\tu_0)+\lim_{n\to\infty}I_F(u_n-\tu_0).
	\end{equation}
	Next we define for $0\leq j < K$
	$$
	r_n^j := u_n - \sum_{i=0}^j \tilde u_i(\cdot-y_n^i)
	$$
	and observe that $r_n^j(\cdot+y_n^{j+1})\to \tilde u_{j+1}$ and $r_n^j(\cdot+y_n^{j+1})-r_n^{j+1}(\cdot+y_n^{j+1})=\tilde u_{j+1}$. 
    Starting with $j=0$ we obtain again, similarly as in proof of Lemma \ref{prop:properties}, that
	$$\lim_{n\to\infty}\big(I_F(r_n^0(\cdot +y_n^1))
	-I_F(r_n^1(\cdot +y_n^1))\big)=I_F(\tu_1)
	$$
	and since $r_n^0=u_n-\tilde u_0$ and by the shift-invariance of $I_F$ this implies 
	\begin{equation} \label{rel2}
	\lim_{n\to\infty}I_F(u_n-\tu_0)=\lim_{n\to\infty}I_F(r_n^0)=
	I_F(\tu_1)+\lim_{n\to\infty}I_F(r_n^1).
	\end{equation}
	Now, for any $0\leq j<K$
	$$\lim_{n\to\infty}\big(I_F(r_n^{j}(\cdot +y_n^{j+1}))
	-I_F(r_n^{j+1}(\cdot +y_n^{j+1}))\big)=I_F(\tu_{j+1})
	$$
	and hence
	\begin{equation} \label{rel3}
	\lim_{n\to\infty}I_F(r_n^j)=
	I_F(\tu_{j+1})+\lim_{n\to\infty}I_F(r_n^{j+1}).
	\end{equation}
	Collecting \eqref{rel1}, \eqref{rel2} and \eqref{rel3} for $0\leq j<K$ we obtain
	$$\lim_{n\to\infty}I_F(u_n)=\sum_{i=0}^{j+1}I_F(\tu_i)+\lim_{n\to\infty}I_F(r_n^{j+1}).$$
    Then by \eqref{eq:Decomp3} 
	$$\lim_{k\to\infty}\lim_{n\to\infty}I_F(r_n^{k})=0$$
	and we
	finally obtain \eqref{eq:Decomp4}.
\end{proof}

\begin{Cor}\label{Cor:weak} The map 
	$J'$ is weak-to-weak$^*$ continuous on $\cM$, i.e. if $(u_n)\subset\cM$ and $u_n\weakto u$ for some $u\in X$, then $J'(u_n)(\vp)\to J'(u)(\vp)$ for any $\vp\in X$. Similarly,  $J_0'$ is weak-to-weak$^*$ continuous on $\cM_0$ where $J_0(u)=\frac{1}{2}\|p_{\cV}(u)\|^2-I_0(u)$ .
\end{Cor}
\begin{proof}
	Since $(u_n)$ is bounded, in view of Theorem~\ref{ThSplit}(c) for $i=0$, up to a subsequence $u_n(x)\to u(x)$ for a.e. $x\in\R^2$. By Vitaly's convergence theorem we infer that $J'(u_n)(\vp)\to J'(u)(\vp)$ and $J'_0(u_n)(\vp)\to J_0'(u)(\vp)$ for any $\vp\in X$.
\end{proof}

\section{Proof of Theorem \ref{th:main} and Theorem \ref{th:main2}}\label{sec:proof}

\begin{Lem}\label{ineq:Nehari}
Suppose that (F5) holds. If $u\in X$, $\psi\in\W$ and $t\geq 0$, then
$$J(u)\geq J(t(u+\psi))+J'(u)\Big(\frac{1-t^2}{2}u-t^2\psi\Big).$$
\end{Lem}
\begin{proof}
We define a map $\varphi:[0,+\infty)\times \R^2\to \R$ such that
\begin{equation}\label{eq:defPhi}
\varphi(t,x):=
\Big\langle f(x,u),\frac{t^2-1}{2}u(x)+t^2\psi(x)\Big\rangle+F(x,u(x))-F(x,t(u(x)+\psi(x)).
\end{equation}
If $u(x)\neq 0$, then (F5) implies that  $\vp(0,x)\leq0$ and if $u(x)=0$ then also $\vp(0,x)=0$. Next, (F4) implies that $\vp(t,x)\to -\infty$ as $t\to\infty$. Then, for every fixed $x\in\R^2$, there is a global maximum point $t_0=t_0(x)>0$ of $\vp(x,\cdot)$. Thus $\vp'(t_0,x)=0$ and by (F5) we obtain that $\vp(t_0,x)\leq 0$; see similar arguments in proof of \cite{MederskiENZ}[Proposition 4.1].  Therefore 
$$J(t(u+\psi))+J'(u)\Big(\frac{1-t^2}{2}u-t^2\psi\Big)- J(u)=-\frac {t^2}{2}\int_{\R^2}V(x)|\psi|^2\,dx+\int_{\R^2}\vp(t,x)\,dx\leq 0$$ and we conclude.
\end{proof}

\begin{Rem} \label{rem:F5}
Observe that the inequality $\vp(1,x)\leq 0$ implies the convexity of $F$ in $u$. Therefore this assumption is not explicitly stated in Theorem \ref{th:main}.
\end{Rem}

\begin{Lem}\label{ineq:AR}
	Suppose that (F6) holds. For any $\eps>0$ there is a constant $c_\eps>0$ such that if $u=v+w\in \cV\oplus \cW$, then
	$$J(u)\geq J(tv)+J'(u)\Big(\frac{1-t^2}{2}u+t^2w\Big)-\eps t^2 |u|_2|w|_2-t^2c_\eps |u|_p^{p-1}|w|_p$$
	for any  $0\leq t\leq \sqrt{1-\frac{2}{\gamma}}$.
\end{Lem}
\begin{proof}
	Observe that by (F6)
	\begin{eqnarray*}
	\varphi(t,x)&=&
	\Big\langle f(x,u),\frac{t^2-1}{2}u\Big\rangle-t^2\langle f(x,u),w\rangle+F(x,u)-F(x,tv),\\
	&\leq&
	-t^2\langle f(x,u),w\rangle-F(x,tv)\leq -t^2\langle f(x,u),w\rangle
	\end{eqnarray*}
	Then we obtain
	\begin{eqnarray*}
	J(tv)+J'(u)\Big(\frac{1-t^2}{2}u+t^2w\Big)- J(u)&=&-\frac {t^2}{2}\int_{\R^2}V(x)|w|^2\,dx+\int_{\R^2}\vp(t,x)\,dx\\
	&\leq& -t^2\int_{\R^2}\langle f(x,u),w\rangle\,dx\\
	&\leq& \eps t^2 |u|_2|w|_2+t^2c_\eps |u|_p^{p-1}|w|_p,
	\end{eqnarray*}
where the last inequality follows from \eqref{ceps} 
	and we conclude.
\end{proof}

\begin{Lem}\label{lem:bounded}
 If (F5), or (F4') and (F6) hold, then $(M)_0^\beta$ is satisfied for every $\beta>0$.
\end{Lem}
\begin{proof}
(a) Take $\beta>0$ and suppose that $u_n=v_n+w_n\in\cM$ is such that $$0\le\liminf_{n\to\infty}J(u_n)\le\limsup_{n\to\infty}J(u_n)\leq\beta\hbox{ and }J'(u_n)(1+\|v_n\|)\to 0$$ as $n\to\infty$.
Observe that  by \eqref{eq:ineqFlp}
$$\frac12\|v_n\|^2-\Big(\frac12\essinf V-\eps\Big)|v_n|_2^2\geq J(u_n)+\Big(\frac12\essinf V-\eps\Big) |w_n|_2^2+\tilde{c}_\eps|u_n|_p^p$$
and if we take $\eps=\frac14\essinf V$ then
\begin{equation}\label{eq:ineqp}
\|v_n\|^2-\frac12\essinf V|v_n|_2^2\geq J(u_n)+\frac12\|v_n\|^2-\frac14\essinf V|v_n|_2^2+\frac14\essinf V|w_n|_2^2+\tilde{c}_\eps|u_n|_p^p.
\end{equation}
Let $s_n:=\Big(\frac12\|v_n\|^2-\frac14\essinf V|v_n|_2^2+\frac14\essinf V|w_n|_2^2+\tilde{c}_\eps|u_n|_p^p\Big)^{1/2}$
and suppose by a contradiction that $(u_n)$ is unbounded, that is, passing to a subsequence $s_n\to\infty$.
Let $\tv_n:=v_n/s_n$ and we may assume that $\tv_n\rightharpoonup \tv$ in $X$ for some $\tv\in\cV$, and $\tv_n(x)\to \tv(x)$ a.e. in $\R^2$. 

Let us show that $\inf_{n\in\N}|\tv_n|_p>0$ by a contradiction argument both in the case where (F5) and (F4'), (F6) holds. Therefore assume (passing to a subsequence) that $\tv_n\to 0$ in $L^p(\R^2)^6$. By \eqref{ceps} we obtain
\begin{equation}\label{eq:convF0}
\int_{\R^2}F(x,s\tv_n)\,dx\to 0\hbox{ for any }s>0.
\end{equation}

First assume that (F5) is satisfied. Take $\gamma\in (0,1)$ such that
\begin{equation}\label{eq:gammadef}
\esssup V\leq \frac12\gamma \essinf V+(1-\gamma)k^2
\end{equation} and $s=2\sqrt{\beta/\gamma}$. In view of Lemma~\ref{ineq:Nehari}, the fact that $J'(u_n)(1+\|v_n\|\to 0$ and \eqref{eq:convF0} we obtain the following estimate 
\begin{eqnarray*}\label{EqIneq2}
\beta&\geq &\liminf_{n\to\infty}J(u_n)\geq 
 \liminf_{n\to\infty} J(s\tv_n)
 +\liminf_{n\to\infty}  J'(u_n)\Big(\frac{1-(s/s_n)^2}{2}u_n+(s/s_n)^2w_n\Big)
 \\
&=& \liminf_{n\to\infty} J(s\tv_n)\geq\frac{s^2}{2}\liminf_{n\to\infty}\big(\|\tv_n\|^2-\esssup V|\tv_n|_2^2\big).
\end{eqnarray*}
Using $\|\tilde v_n\|-\frac{1}{2}\essinf V |\tilde v_n|_2\geq 1$ by \eqref{eq:ineqp} and \eqref{eq:gammadef}  we obtain the contradiction
\begin{eqnarray*}
\beta 
&\geq& \frac{s^2}{2}\gamma\liminf_{n\to\infty}\Big(\|\tv_n\|^2-\frac12\essinf V|\tv_n|_2^2\Big)\geq 
\frac{s^2}{2}\gamma=2\beta.
\end{eqnarray*}

Next, assume that (F4') and (F6) are satisfied. Then
$$J(u_n)-\frac12J'(u_n)(u_n)\geq 
\frac{\gamma-2}{2}
\int_{\R^2}F(x,u_n)\, dx$$
and $(u_n)$ is bounded in $L^p(\R^3)^6$. Hence $(w_n)$ is bounded in $L^p(\R^3)^6$. 
Note that $s_n\geq \delta |u_n|_2$ for some constant $\delta>0$. 
Hence also $s_n\geq \delta |w_n|_2$  and taking $\gamma\in (0,1)$ as in \eqref{eq:gammadef}, $s=2\sqrt{\beta/\gamma}$ and $\eps=\frac18\gamma\delta^2$, in view of Lemma \ref{ineq:AR}, the fact that $J'(u_n)(1+\|v_n\|\to 0$ and \eqref{eq:convF0} we obtain the following estimate
\begin{eqnarray*}
	\beta&\geq &\liminf_{n\to\infty}J(u_n)\geq 
	\liminf_{n\to\infty} J(s\tv_n)
	+\liminf_{n\to\infty}  J'(u_n)\Big(\frac{1-(s/s_n)^2}{2}u_n+(s/s_n)^2w_n\Big)\\
	&&+\liminf_{n\to\infty} \big(-\eps (s/s_n)^2 |u_n|_2|w_n|_2-(s/s_n)^2c_\eps |u_n|_p^{p-1}|w_n|_p\big)\\
	&\geq& \liminf_{n\to\infty} J(s\tv_n)-\eps s^2\delta^{-2} =\frac{s^2}{2}\liminf_{n\to\infty}\big(\|\tv_n\|^2-\esssup V|\tv_n|_2^2\big)-\eps s^2\delta^{-2}.
\end{eqnarray*}	
Using as before $\|\tilde v_n\|-\frac{1}{2}\essinf V |\tilde v_n|_2\geq 1$, the definition \eqref{eq:gammadef} of $\gamma$ and the definition of $\epsilon$ we obtain the contradiction
\begin{eqnarray*}
	\beta&\geq& \frac{s^2}{2}\gamma\liminf_{n\to\infty}\Big(\|\tv_n\|^2-\frac12\essinf V|\tv_n|_2^2\Big)-\eps s^2\delta^{-2}\geq s^2\Big(
	\frac{\gamma}{2}-\eps \delta^{-2}\Big)=\frac32\beta.
\end{eqnarray*}

In both cases we have led the assumption that $\tv_n\to 0$ in $L^p(\R^2)^6$ to a contradiction. Hence, we continue by having $\inf_{n\in\N}|\tv_n|_p>0$. Let $\cl_{L^p}\cV$ denotes the closure of $\cV$ in $L^p(\R^2)^6$. Observe that $\cl_{L^p}\cV\cap \cW=\{0\}$. Using the continuity of the projection $\cl_{L^p}\cV\oplus \cW\to \cl_{L^p}\cV$ there is a constant $c'>0$ such that
\begin{equation}\label{projectionL^p}
|u_n|_p^p\geq c' |v_n|_p^p.
\end{equation}
Then by \eqref{eq:ineqFlp} we get
$$\int_{\R^2}F(x,u_n)\,dx\geq -\eps |u_n|_2^2+\tilde{c}_\eps|u_n|_p^p\geq -\eps |u_n|_2^2+\tilde{c}_\eps c' |v_n|_p^p.$$
Therefore taking $\eps=\frac{1}{2}\essinf V$ we see that
\begin{eqnarray*}
	J(u_n)/s_n^2&\leq&\frac{1}{2}\|\tv_n\|^2-\frac{1}{2s_n^2}\essinf V|u_n|_2^2-
	\int_{\R^2}\frac{F(x,u_n(x))}{s_n^2}\,dx\\
&\leq&\frac{1}{2}\|\tv_n\|^2-
\tilde{c}_\eps c' s_n^{p-2}|\tv_n|_p^p\leq\frac{1}{2}\|\tv_n\|^2-
\tilde{c}_\eps c' s_n^{p-2}\inf_{n\in\N}|\tv_n|_p^p\\\
	&\to&-\infty.
\end{eqnarray*}
Thus we get a contradiction, and therefore $(u_n)$ must be bounded so that part (a) of $(M)_0^\beta$ holds.

Now let us show part (b) of $(M)_0^\beta$. This part only applies in the case where $V(x)\equiv V_0$. So let us take another sequence 
$u_n'=v_n'+w_n'\in\cM$ is such that 
$$0\le\liminf_{n\to\infty}J(u_n')\le\limsup_{n\to\infty}J(u_n')\leq\beta\hbox{ and }J'(u_n')(1+\|v_n'\|)\to 0$$ as $n\to\infty$.
Suppose in addition that the number of critical orbits in $J_0^\beta$ is finite. 
Let
$$m:=\inf\big\{\|v-v'\|: J'(m(v))=J'(m(v'))=0, v\neq v'\big\}$$
and since $G=\Z^2$ is discrete (i.e., condition (G) holds), $m>0$. Now we split the proof into two cases.

\medskip

Case 1: $|v_n-v_n'|_p\to 0$. Then 
\begin{eqnarray*}
\|v_n-v_n'\|^2-\esssup V |v_n-v_n'|_2^2&\leq&\big(J'(u_n)-J'(u_n')\big)(u_n-u_n')\\
&&+\int_{\R^2}\langle f(x,u_n)-f(x,u_n'),v_n-v_n'\rangle\,dx.
\end{eqnarray*}
Since $(u_n), (u_n')$ are bounded Palais-Smale sequences the first term on the right hand side converges to $0$ and by \eqref{ceps} the last integral on the right hand side tends to $0$. Hence, using $\esssup V<k^2$ we see that $\|v_n-v_n'\|^2\to 0$ as $n\to\infty$ and the proof is finished. 

\medskip

Case 2: $\limsup |v_n-v_n'|_p>0$. Then by Lion's Lemma \cite[Lemma 1.21]{Willem}, up to a $\Z^2$ translation, $v_n\weakto v$ and $v_n'\weakto v'$ for some $v,v'\in\cV$ such that $v\neq v'$. We may assume that $w_n\weakto w$ and $w_n'\weakto w'$ for some $w,w'\in\cW$. In view of Corollary \ref{Cor:weak}, $J'(v+w)=J'(v'+w')=0$ and thus $\|v-v'\|\geq m$. In other words: if $\|v-v'\|< m$ then we are in the previous Case 1 and the proof is completed.
\end{proof}

\begin{altproof}{Theorem \ref{th:main}}
Suppose that $(u_n)\subset \cM$  sequence obtained in Theorem \ref{ThLink1}. By Lemma~\ref{lem:bounded} it is a bounded Palais-Smale sequence for the functional $J$. Passing to a subsequence, we find $\tu_0\in \cM$ and sequences $(\tu_i)_{i\geq 0}\subset \cM_0$, $(y_n^i)_{n\geq i}\subset \Z^2$ such that the statements of Theorem~\ref{ThSplit} hold. Observe that by Corollary \ref{Cor:weak},  $\tu_0$ is a critical point of $J$. Similarly we show that $\tu_i$ is a critical point of $J_0$ for $i\geq 1$. Indeed, take any $\vp\in X$ and since $|y_n^i|\to \infty$ and  $V-V_0\in L^{\frac{p}{p-2}}(\R^2)$ we obtain $|V(\cdot+y_n^i)-V_0|_{\frac{p}{p-2}}\to 0$  as $n\to\infty$ and
\begin{eqnarray*}
J'_0(u_n(\cdot+y_n^i))(\vp)&=&J'_0(u_n)(\vp(\cdot-y_n^i))=
J'(u_n)(\vp(\cdot-y_n^i))+\int_{\R^2}(V-V_0)\langle u_n, \vp(\cdot-y_n^i)\rangle \,dx\\
&=&J'(u_n)(\vp(\cdot-y_n^i))+\int_{\R^2}(V(\cdot+y_n^i)-V_0)\langle u_n(\cdot+y_n^i), \vp\rangle \,dx = o(1).
\end{eqnarray*}
Since 
$u_n(x+y_n^i)\to \tu_i(x)$ for a.e. $x\in\R^2$ by Vitaly's convergence theorem we infer that $ J_0'(\tu_i)(\vp)=0$.
We may assume that $\lim_{n\to\infty}I_F(u_n)$ exists and is positive, since otherwise (for a subsequence) $u_n\to 0$ in $L^p(\R^2)^6$
and
$$J(u_n)=J(u_n)-\frac12J'(u_n)(u_n) +o(1)=\int_{\R^2}\frac12 \langle f(x,u_n),u_n\rangle-F(x,u_n)\,dx=o(1)$$
we get a contradiction.
 Therefore, having $\lim_{n\to\infty}I_F(u_n)>0$, we know that  $\tu_0\neq 0$ or $\tu_1\neq 0$.

\medskip

We first treat the case where $V=V_0$, hence $J=J_0$. Due to shift-invariance in this case,  we may assume w.l.o.g. that $\tu_0\neq 0$. Clearly, by Fatou's lemma
\begin{equation*}
	\inf_{\cN_0}J_0=\lim_{n\to\infty}J_0(u_n)=\lim_{n\to\infty}\Big(J_0(u_n)-\frac12J_0'(u_n)(u_n)\Big)\geq  J_0(\tu_0)-\frac12J'_0(\tu_0)(\tu_0)=J_0(\tu_0),
\end{equation*}
thus $J_0(\tu_0)=\inf_{\cN_0}J_0$ and $\tu_0\in\cN_0$ is a ground state solution.  If, in addition, $f$ is odd in $u$, then in view of Lemma \ref{lem:bounded} and Theorem \ref{Th:CrticMulti} there is an infinite sequence of $\Z^2$-distinct solutions.

\medskip

Next we consider the case that $V(x)>V_0$ for a.e. $x\in\R^2$ and (F4') is satisfied. Again we apply Theorem \ref{ThSplit}. Let $v_n:=p_{\cV}(u_n)$ and $\tv_i:=p_{\cV}(\tu_i)$ for $n\geq 1$ and $0\leq i<K+1$.
In view of \eqref{eq:Decomp1} from Theorem~\ref{ThSplit}, for any $0\leq k<K+1$ we get
$$\lim_{n\to\infty}\|v_n\|^2\geq \sum_{i=0}^{k}\|\tv_i\|^2
\geq 2J(\tu_0)+\sum_{i=1}^{k}2J_0(\tu_i)\geq 2k\inf_{\cN_0}J_0,$$
hence  $K<\infty$. Observe that 
\begin{eqnarray*}
\lefteqn{
J'(u_n)\Big[v_n-\sum_{i=0}^K\tv_i(\cdot-y_n^i)\Big]} \\
&=&\Big\|v_n-\sum_{i=0}^K\tv_i(\cdot-y_n^i)\Big\|^2+
 b_L\Big(\sum_{i=0}^K\tv_i(\cdot-y_n^i), v_n-\sum_{i=0}^K\tv_i(\cdot-y_n^i)\Big)\\
 &&-I'(u_n)\Big[v_n-\sum_{i=0}^K\tv_i(\cdot-y_n^i)\Big]\\
 &=&\Big\|v_n-\sum_{i=0}^K\tv_i(\cdot-y_n^i)\Big\|^2+
 \sum_{j=0}^Kb_L\Big(\tv_j, v_n(\cdot+y_n^j)-\sum_{i=0}^K\tv_i(\cdot-(y_n^i-y_n^j))\Big)\\
 &&-I'(u_n)\Big[v_n-\sum_{i=0}^K\tv_i(\cdot-y_n^i)\Big]
\\
 &=&\Big\|v_n-\sum_{i=0}^K\tv_i(\cdot-y_n^i)\Big\|^2- I'(u_n)\Big[v_n-\sum_{i=0}^K\tv_i(\cdot-y_n^i)\Big]+o(1),
\end{eqnarray*}
since $v_n(\cdot+y_n^j)\weakto \tv_j$, and for $i\neq j$, $b_{L}(\tv_j,\tv_i(\cdot-(y_n^i-y_n^j)))\to 0$ as $n\to\infty$. 
Since (F4') holds, then by \eqref{eq:Decomp4} we get
 $u_n-\sum_{i=0}^K\tu_i(\cdot-y_n^i)\to 0$ in $L^p(\R^2)^6$, hence  $v_n-\sum_{i=0}^K\tv_i(\cdot-y_n^i)\to 0$ in $L^p(\R^2)^6$.
Then
\begin{eqnarray*}
I'(u_n)\Big[v_n-\sum_{i=0}^K\tv_i(\cdot-y_n^i)\Big]&=& V_0\Big|v_n-\sum_{i=0}^K\tv_i(\cdot-y_n^i)\Big|_2^2\\
&&+\sum_{j=0}^K\int_{\R^2}V_0\Big\langle \tv_j, v_n-\sum_{i=0}^K\tv_i(\cdot-(y_n^i-y_n^j)) \Big\rangle\,dx\\
&&+\int_{\R^2}(V(x)-V_0)\Big\langle v_n, v_n-\sum_{i=0}^K\tv_i(\cdot-y_n^i)\Big\rangle\\
&&+I'_F(u_n)\Big[v_n-\sum_{i=0}^K\tv_i(\cdot-y_n^i)\Big]\\
&=&V_0\Big|v_n-\sum_{i=0}^K\tv_i(\cdot-y_n^i)\Big|_2+o(1)
\end{eqnarray*}
and we get
\begin{eqnarray*}
J'(u_n)\Big[v_n-\sum_{i=0}^K\tv_i(\cdot-y_n^i)\Big]&=& \Big\|v_n-\sum_{i=0}^K\tv_i(\cdot-y_n^i)\Big\|^2-
V_0\Big|v_n-\sum_{i=0}^K\tv_i(\cdot-y_n^i)\Big|_2^2+o(1)\\
&\geq & \frac{k^2-V_0}{k^2} \Big\|v_n-\sum_{i=0}^K\tv_i(\cdot-y_n^i)\Big\|^2+o(1)
\end{eqnarray*}
and since $(u_n)$ is a bounded Palais-Smale sequence we obtain 
$$v_n-\sum_{i=0}^K\tv_i(\cdot-y_n^i)\to0\quad\hbox{in } H^1(\R^2)^6.$$
Arguing similarly as  above we get
\begin{eqnarray*}
	I'(u_n)\Big[u_n-\sum_{i=0}^K\tu_i(\cdot-y_n^i)\Big]&=& V_0\Big|u_n-\sum_{i=0}^K\tu_i(\cdot-y_n^i)\Big|_2^2+o(1)
\end{eqnarray*}
and 
$$u_n-\sum_{i=0}^K\tu_i(\cdot-y_n^i)\to0\quad\hbox{in } L^2(\R^2)^6.$$
Hence, using the divergence property $|y_n^i-y_n^j|\to \infty$ as $n\to \infty$ for $i\not= j$, we get
\begin{eqnarray*}
	\lim_{n\to\infty}\|v_n\|^2&=&\sum_{i=0}^{K}\|\tv_i\|^2,\\
	\lim_{n\to\infty}|u_n|^2_2&=&\sum_{i=0}^{K}|\tu|_2^2,
\end{eqnarray*}
and in view of \eqref{eq:Decomp3}  we finally get 
\begin{equation*}
c=\inf_{\cN}J=\lim_{n\to\infty}J(u_n)=J(\tu_0)+\sum_{i=1}^{K}J_0(\tu_i)\geq  
J(\tu_0) +K\inf_{\cN_0}J_0
\end{equation*}

Let us denote by $\xi$ the ground state solutions of $J_0$ obtained above. Observe that $J(t\xi+w(t\xi))\to -\infty$ as $t\to\infty$, $J(0+w(0))=0$, hence  we find $t>0$ such that
	$$J(t\xi+w(t\xi))\geq c,$$ where $c$ is given by \eqref{cmX}.
	Then, in view of Lemma \ref{ineq:Nehari} and Theorem \ref{ThLink1}
	$$\inf_{\cN_0}J_0=J_0(\xi)\geq J_0(t\xi+w(t\xi))>J(t\xi+w(t\xi))\geq c=\inf_{\cN}J.$$
	Therefore $K=0$ and $\tu_0\neq 0$ is a critical point of $J$. Observe that $\tu_0\in\cN$ and (using once more Fatou's Lemma)
	$$
		c=\lim_{n\to\infty}J(u_n)=\lim_{n\to\infty}\big(J(u_n)-\frac12J'(u_n)(u_n)\big)\geq \lim_{n\to\infty} J(\tu_0)-\frac12J'(\tu_0)(\tu_0)=J(\tu_0)
	$$
	so that $J(\tu_0)=\inf_{\cN}J$ and $\tu_0$ is a ground state solution.
\end{altproof}

\begin{altproof}{Theorem \ref{th:main2}}
 Here $V=V_0$ and we argue as in proof of Theorem \ref{th:main} and we use the notation introduced therein.
Suppose that $(u_n)\subset \cM_0$ is the bounded Palais-Smale sequence obtained in Section \ref{sec:PS}. Passing to a subsequence, we find sequences $(\tu_i)_{i\geq 0}\subset \cM_0$, $(y_n^i)_{n\geq i}\subset \Z^2$ such that all the statements of Theorem~\ref{ThSplit} hold. Observe that by Corollary \ref{Cor:weak},  $\tu_0$ is a critical point of $J$ and $\tu_i$ is a critical point of $J_0$.
Since $\lim_{n\to\infty}I_F(u_n)>0$, $\tu_0\neq 0$ or $\tu_1\neq 0$. Note that
$$\cK_0:=\big\{u\in X\setminus\{0\}: J'_0(u)=0\big\}\subset\cM_0$$
and we already know that $\cK_0\neq \emptyset$ since $\tilde u_0\not=0$ or $\tilde u_1\not=0$. It is easy to show that $c_0:=\inf_{\cK_0}J_0>0$ and, since $\cK$ is nonempty, we may take any minimizing sequence $(\xi_n)\subset\cK_0$ such that $J_0(\xi_n)\to c_0$. Similarly as above applying the profile decomposition from Theorem~\ref{ThSplit} to $(\xi_n)$, passing to a subsequence we find $(z_n)\subset\Z^2$ such that $\xi_n(\cdot+z_n)\weakto \xi\in\cK_0$ and $\xi_n(\cdot+z_n)\to \xi$ a.e. on $\R^2$.
Due to Fatou's Lemma we have 
$$
\lim_{n\to\infty}J_0(\xi_n)=\lim_{n\to\infty} \big(J_0(\xi_n)-\frac12J_0'(\xi_n)(\xi_n)\big)\geq J_0(\xi)-\frac12J_0'(\xi)(\xi)=J_0(\xi),
$$
$J_0(\xi)=c_0$ and we conclude. If, in addition, $f$ is odd in $u$, then in view of Lemma \ref{lem:bounded} and Theorem \ref{Th:CrticMulti} there is an infinite sequence of $\Z^2$-distinct solutions.
\end{altproof}

{\bf Acknowledgements.} Funded by the Deutsche Forschungsgemeinschaft (DFG, German Research Foundation) -- Project-ID 258734477 -- SFB 1173. J. Mederski was also partially supported by the Alexander von Humboldt Foundation (Germany) during the stay at Karlsruhe Institute of Technology and by the National Science Centre, Poland (Grant No. 2020/37/B/ST1/02742). He would like to express his deep gratitude to these institutions for their support and warm hospitality.



\begin{thebibliography}{99}
\baselineskip 2 mm

\bibitem{Akhmediev-etal} N.N. Akhmediev, A. Ankiewicz, J.M. Soto-Crespo: {\em Does the nonlinear Schr\"odinger equation correctly describe beam propagation?}, Opt. Lett. {\bf 18} (1993), 411.

\bibitem{AR} A. Ambrosetti, P.H. Rabinowitz, {\em Dual variational methods in critical point theory and applications}, J. Funct.  Anal. {\bf 14} (1973), 349--381. 

\bibitem{BDPR} T. Bartsch, T. Dohnal, M. Plum, W. Reichel: {\em Ground states of a nonlinear curl-curl problem in cylindrically symmetric media}, NoDEA Nonlinear Diff. Equ. Appl. {\bf 23} (2016)(5), Art. 52, 34 pp.

\bibitem{BartschMederski} T. Bartsch, J. Mederski: {\em Ground and bound state solutions of semilinear time-harmonic Maxwell equations in a bounded domain}, Arch. Rational Mech. Anal. {\bf 215} (1), (2015), 283--306.

\bibitem{BartschMederskiJFA} T. Bartsch, J. Mederski: {\em Nonlinear time-harmonic Maxwell equations in an anisotropic bounded medium}, J. Funct. Anal. {\bf 272} (2017), no. 10, 4304--4333. 

\bibitem{BartschMederskiSurvey} T. Bartsch, J. Mederski: {\em Nonlinear time-harmonic Maxwell equations in domains}, J. Fixed Point Theory Appl. 19 (2017), no. 1, 959--986.

\bibitem{bbf} P. Bartolo, V. Benci, D. Fortunato, \emph{Abstract critical point theorems and applications to some nonlinear problems with ``strong'' resonance at infinity},  Nonlinear Anal. \textbf{7} (1983), 981--1012.

\bibitem{BenciRabinowitz} V. Benci, P. H. Rabinowitz: {\em Critical point theorems for indefinite functionals}, Invent. Math. {\bf 52} (1979), no. 3, 241--273.


\bibitem{BrezisLieb} H. Br\'ezis, E. Lieb: {\em A relation between pointwise convergence of functions and convergence of functionals}, Proc. Amer. Math. Soc. {\bf 88} (1983), no. 3, 486--490.

\bibitem{Cerami}  G. Cerami: {\em An existence criterion for the critical points on unbounded manifolds}, Istituto Lombardo. Accademia di Scienze e Lettere. Rendiconti. Scienze Matematiche, Fisiche, Chimiche e Geologiche. A, vol. 112, no. 2, pp. 332--336, (1978) (Italian).


\bibitem{Ciattoni-etal:2005} A. Ciattoni, B. Crossignani, P. Di Porto, A. Yariv: {\em Perfect optical solitons: spatial Kerr solitons as exact solutions of Maxwell's equations}, J. Opt. Soc. Am. B {\bf 22} (2005), 1384--94.


\bibitem{Gerard} G\'erard: {\em Description du d\'efaut de compacit\'e de l'injection de Sobolev}, ESAIM: Control, Optimisation and Calculus of Variations {\bf 3} (1998), 213--233. 

\bibitem{HmidiKeraani} T. Hmidi, S. Keraani: {\em Remarks on the blow-up for the $L^2$-critical nonlinear Schr\"odinger equations}, SIAM J. Math. Anal. {\bf 38} (2006), no. 4, 1035--1047.


\bibitem{McLeodStuartTroy} J.B. McLeod, C.A. Stuart, W.C. Troy: {\em An exact reduction of Maxwell's equations},  {\em Nonlinear diffusion equations and their equilibrium states}, 3 (Gregynog, 1989), 391--405, Progr. Nonlinear Differential Equations Appl. {\bf 7}, Birkh\"auser Boston, (1992). 

\bibitem{MederskiENZ} J. Mederski: {\em Ground states of time-harmonic semilinear Maxwell equations in $\R^3$ with vanishing permittivity}, Arch. Rational Mech. Anal. {\bf 218} (2), (2015), 825--861.

\bibitem{MederskiSystem} J. Mederski: {\em Ground states of a system of nonlinear Schr\"odinger equations with periodic potentials}, Comm. Partial Differential Equations {\bf 41} (9), (2016), 1426--1440.


\bibitem{MederskiBL} J. Mederski: {\em Nonradial solutions of nonlinear scalar field equations}, Nonlinearity {\bf 33} (2020), 6349--6380.

\bibitem{MederskiSchinoSzulkin} J. Mederski, J. Schino, A. Szulkin: {\em Multiple solutions to a nonlinear curl-curl problem in $\mathbb{R}^3$}, Arch. Rational Mech. Anal. 236 (2020) 253--288.


\bibitem{Nawa} H. Nawa: {\em  Mass concentration phenomenon for the nonlinear Schr\"odinger equation with the critical power nonlinearity. II } Kodai Math. J. {\bf 13} (1990), no. 3, 333--348.

\bibitem{Nehari2} Z. Nehari: {\em Characteristic values associated with a class of non-linear second-order differential equations}, Acta Math. {\bf 105} (1961), 141--175.


\bibitem{Pankov} A. Pankov: {\em Periodic Nonlinear Schr\"odinger Equation with Application to Photonic Crystals}, Milan J. Math. {\bf 73} (2005), 259--287.


\bibitem{FundPhotonics} B.E.A. Saleh, M.C. Teich: {\em Fundamentals of Photonics}, 2nd Edition, Wiley 2007.

\bibitem{Stuart91} C. A. Stuart:
{\em Self-trapping of an electromagnetic field and bifurcation from the essential spectrum}, Arch. Rational Mech. Anal. {\bf 113} (1991), no. 1, 65--96.

\bibitem{Stuart:1993} C. A. Stuart: {\em Guidance Properties of Nonlinear Planar Waveguides}, Arch. Rational Mech. Anal. {\bf 125} (1993), no. 1, 145--200.

\bibitem{Stuart04} C. A. Stuart:
{\em Modelling axi-symmetric travelling waves in a dielectric with nonlinear refractive index},
Milan J. Math. {\bf 72} (2004), 107--128.

\bibitem{StuartZhou96}
C.A. Stuart, H.S. Zhou:
{\em A variational problem related to self-trapping of an electromagnetic field},
Math. Methods Appl. Sci. {\bf 19} (1996), no. 17, 1397--1407.

\bibitem{StuartZhou01} C.A. Stuart, H.S. Zhou:
{\em Existence of guided cylindrical TM-modes in a homogeneous self-focusing dielectric}, Ann. Inst. H. Poincar\'e Anal. Non Lin\'eaire {\bf 18} (2001), no. 1, 69--96.

\bibitem{StuartZhou03} C.A. Stuart, H.S. Zhou:
{\em A constrained minimization problem and its application to guided cylindrical TM-modes in an anisotropic self-focusing dielectric},
Calc. Var. Partial Differential Equations {\bf 16} (2003), no. 4, 335--373.

\bibitem{StuartZhou05} C.A. Stuart, H.S. Zhou:
{\em Axisymmetric TE-modes in a self-focusing dielectric}, SIAM J. Math. Anal. {\bf 37} (2005), no. 1, 218--237.

\bibitem{StuartZhou10} C.A. Stuart, H.S. Zhou:
{\em Existence of guided cylindrical TM-modes in an inhomogeneous self-focusing dielectric}, Math. Models Methods Appl. Sci. {\bf 20} (2010), no. 9, 1681--1719.

\bibitem{Willem} M. Willem: {\em Minimax Theorems}, Birkh\"auser Verlag (1996).

\bibitem{stein} E. Stein: {\em Singular integrals and differentiability properties of functions},
Princeton Mathematical Series, No. 30 Princeton University Press, Princeton, N.J. (1970). 

\bibitem{reed_simon_vol1} M. Reed, B. Simon: {\em Methods of modern mathematical physics. I.
	Functional analysis.} Second edition. Academic Press, New York (1980).

\end{thebibliography}
\end{document}